\documentclass[10pt]{article}

\usepackage[top=1in,bottom=1in,left=1.25in,right=1.25in]{geometry}

\usepackage[absolute,overlay]{textpos}

\usepackage{graphicx}
\usepackage{ragged2e}
\usepackage{tabto}

\usepackage{amsmath}
\usepackage{amsthm}
\usepackage{amsfonts}

\usepackage{float}
\usepackage{multirow}
\usepackage{tikz}
\usepackage{graphics}
\usepackage{color}
\usepackage{url}
\usepackage{afterpage}
\usepackage{multicol}
\usepackage[font=footnotesize]{caption}
\usepackage[font=footnotesize]{subcaption}
\usepackage{marginnote}
\usepackage{tcolorbox}

\usepackage{lmodern}

\usepackage[sf,bf,medium]{titlesec}
\usepackage{pgfplots}

\usepackage{bm}
\usepackage{booktabs}
\usepackage{siunitx}
\usepackage{isomath}
\usepackage{algorithm}

\usetikzlibrary{arrows}
\usetikzlibrary{shapes}
\usetikzlibrary{calc}
\usetikzlibrary{patterns}
\tikzset{
    partial ellipse/.style args={#1:#2:#3}{
        insert path={+ (#1:#3) arc (#1:#2:#3)}
    }
}

\usepackage[plainpages=false, colorlinks=true, citecolor=blue, filecolor=blue,
   linkcolor=blue, urlcolor=blue]{hyperref}
\usepackage{enumitem}

\usepackage{multicol}
\newcommand\numberthis{\addtocounter{equation}{1}\tag{\theequation}}

\newcommand{\lp}{\left(}
\newcommand{\rp}{\right)}

\usepackage[sort,compress]{cite}

\newcommand\np{I}

\newcommand\bx{\boldsymbol{x}}

\newcommand\by{\boldsymbol{y}}

\newcommand\br{\boldsymbol r}

\newcommand\bn{\boldsymbol n}

\newcommand\bu{\boldsymbol u}
\newcommand\bv{\boldsymbol v}

\newtheorem{theorem}{\sffamily Theorem}

\newtheorem{remark}{\sffamily Remark}
\newtheorem{definition}{\sffamily Definition}

\newtheorem{lemma}{\sffamily Lemma}
\newtheorem{corollary}{\sffamily Corollary}[theorem]

\newcommand{\cP}{\mathcal P}

\newcommand{\cC}{\mathcal C}

\newcommand{\cK}{\mathcal K}

\newcommand{\surfdiv}{\nabla_\Gamma \cdot}
\newcommand{\surfgrad}{\nabla_\Gamma}
\newcommand{\gradg}{\surfgrad}
\newcommand{\divg}{\surfdiv}

\numberwithin{equation}{section}

\newcommand{\bs}{\boldsymbol}

\newcommand{\figref}[1]{\figurename~\hyperref[#1]{\ref{#1}}}
\newcommand{\tableref}[1]{\tablename~\hyperref[#1]{\ref{#1}}}

\renewcommand{\phi}{\varphi}

\binoppenalty=\maxdimen
\relpenalty=\maxdimen

\usepackage{setspace}

\begin{document}
{\centering
{\huge A parametrix for the surface Stokes equation}\\
\vspace{.1cm}

 Tristan Goodwill\footnote{University of Chicago, Chicago, IL 
  (\texttt{tgoodwill@uchicago.edu})}, Jeremy Hoskins\footnote{Department of Statistics and Committee on Computational and Applied Mathematics, University of Chicago, USA and NSF-Simons National Institute for Theory and Mathematics in Biology, Chicago, IL
  (\texttt{jeremyhoskins@uchicago.edu})}, {Zydrunas Gimbutas\footnote{National Institute of Standards and Technology, Boulder, CO (\texttt{zydrunas.gimbutas@nist.gov})}, and Bowei Wu\footnote{University of Massachusetts Lowell, Lowell, MA (\texttt{bowei\_wu@uml.edu})}}\\
} 

\bigskip
\textbf{Abstract}: We introduce an integral equation formulation of the surface Stokes equations, constructed using two-dimensional Stokeslets. The resulting integral equations are Fredholm integral equations of the second kind and can be discretized to high order using standard tools. Since the resulting discrete linear systems are dense, we describe and analyze a \emph{proxy shell method} to construct fast direct solvers for these systems. The properties of our integral equation, and the performance of the resulting numerical scheme, are illustrated with several representative numerical examples.

\smallskip



\section{Introduction}
The flow of thin fluid films occurs in foams, emulsions, and biological applications \cite{scriven1960dynamics,slattery2007interfacial,arroyo2009relaxation,brenner2013interfacial,rangamani2013interaction,rahimi2013curved}. In most such problems, the flow is governed by the Stokes (or Navier-Stokes) equations posed on a manifold embedded in three dimensions. Recently, there has been growing interest in developing numerical methods for solving surface Stokes equations and other surface fluid flow problems. These methods are predominantly finite element methods (FEM) \cite{reusken,brandner2022finite,demlow2024tangential,scriven1960dynamics,kilicer2025higher,neilan2025c0,demlow2025taylor,nitschke2012finite,gross2018hydrodynamic,brandner2020finite,olshanskii2018finite,fries2018higher,reuther2018solving,olshanskii2019penalty,hardering2025parametric,jankuhn2018incompressible,hardering2023tangential}, though there also exist closest point methods which use finite difference discretizations \cite{yang2020practical}. In the FEM approach, the problem is discretized in one of two ways: either the surface is discretized using a \emph{surface finite element method} (SFEM) with a fitted mesh living on the surface (e.g. \cite{demlow2025taylor}); or, in the \emph{trace finite element method} (TraceFEM), an unfitted mesh is constructed in the ambient space around the surface and surface fields are expressed as traces of functions defined on the unfitted mesh (e.g. \cite{olshanskii2018finite}). Due to the difficulties associated with building high order surface meshes and conforming finite elements, implementations of these methods have generally been limited to at most third order \cite{brandner2022finite}, though a fourth order method was recently implemented \cite{kilicer2025higher}. Another contributing factor to the difficulty of this class of problems is the condition that the fluid flow on a static surface must be tangential. This tangency is usually enforced through a Lagrange multiplier or a penalty term \cite{fries2018higher,gross2018trace,brandner2022finite,hardering2025parametric,jankuhn2018incompressible,reuther2018solving,olshanskii2019penalty,hardering2023tangential}, resulting in larger or poorly-conditioned linear systems. Alternatively, there are several methods that build this tangency directly into the approximation space~\cite{demlow2024tangential,demlow2025taylor,kilicer2025higher}. A different approach is to introduce a stream function formulation of the surface Stokes equations \cite{nitschke2012finite,reusken,neilan2025c0,gross2018hydrodynamic,brandner2020finite}, which automatically enforces tangency and divergence-free conditions at the cost of increasing the order of the PDE. See \cite{brandner2022finite,neilan2025c0} for more detailed comparisons of these methods.

In this work, we take a different approach, reducing the surface Stokes equations to a system of integral equations, which we then discretize with a high order collocation scheme. Conceptually, we follow the general  procedure of~\cite{goodwill2024parametrix}, which describes a method for scalar elliptic equations. In essence, the method amounts to constructing suitable \emph{analytic} right-preconditioners for the equations.
Analytic preconditioner (or parametrix) methods have previously been used to solve Stokes problems in Euclidean space, including variable viscosity problems \cite{malhotra2014volume} and the compressible Stokes equations \cite{ayele2021boundary}. More generally, there is a long history of applying related methods to other problems including the Lippmann–Schwinger equation for Helmholtz problems with variable wavenumber~\cite{colton1998inverse,martinsson2019fast}, and variable coefficient Maxwell's equations~\cite{imbert2019integral}.

In the present paper, the velocity field is represented as an integral over the surface of the two-dimensional Stokeslet against an unknown density. Adding a term to account for the incompressibility conditions and inserting this representation into the surface Stokes equations gives a Fredholm second kind integral equation that can be solved for the unknown density. This approach has several advantages. First, the tangential conditions are enforced directly by the integral equation. Second, the nonlocal nature of the integral equations remove many of the difficulties of mesh generation, enabling the construction of higher order discretizations. In particular, the meshes do not need to be conforming, nor even watertight.
Finally, the Fredholm structure of the integral equations guarantees that the discretized equations will be well-conditioned, even when adaptive refinement is used to resolve small-scale features. 

On the other hand, the resulting integral equations give rise to dense matrices when discretized. For large problems this can become prohibitively expensive both in terms of the memory requirements and overall computation time. This feature is common to all integral equation methods and has necessitated the construction of a number of fast algorithms. Many of these algorithms are based on off-diagonal rank structure of the system matrix. We prove that our integral equation has the same rank structure, which enables the use of fairly standard fast algorithms, as well as the construction of fast direct solvers based on the approaches of~\cite{sushnikova2023fmm,minden2017recursive,ho2012fast,l2016hierarchical}. We test our solvers on a variety of examples. This rank structure is of independent interest, since our analysis applies to a much broader class of equations, in particular to the equations developed for scalar surface PDEs in~\cite{goodwill2024parametrix}.

The remainder of the paper is organized as follows. In Section \ref{sec:bckgrnd} we define the surface Stokes equations and briefly review the relevant surface differential operators. Following this, in Section \ref{sec:int_rep} we describe the derivation of our integral equation formulation of the surface Stokes equation and establish several key analytic properties. In Section \ref{sec:disc} we summarize the procedure we use to discretize our system of integral equations, and in Section \ref{sec:accel} we present and analyze a method for accelerating the solution of our discretized system. Finally, in Section \ref{sec:numill} we demonstrate the performance of our method with several numerical examples.

\section{The surface Stokes equation}\label{sec:bckgrnd}
Before stating the surface Stokes equations, we first review the definitions of certain differential operators on smooth, bounded surfaces without boundary. There are a number of equivalent definitions of these operators, and we give the one most convenient for our purposes.
\begin{definition}\label{def:1}
    Let~$\Gamma$ be a smooth and bounded surface. Let~$\bs n$ be the normal to~$\Gamma$, defined as the gradient of a signed distance function. Let~$f$,~$\bs v$, and~$A$ be a function, vector field, and 3-matrix-valued function defined on~$\Gamma$. Let~$f^e$ be a function defined in neighborhood of~$\Gamma$ that is equal to $f$ on $\Gamma$. Let~$\bs v^e$ and~$A^e$ be defined similarly.
    We define the following quantities \cite{reusken}.
 \begin{enumerate}
    \item The projection onto the tangent bundle is
    \begin{equation}
        \cP:= I - \bn \bn^T.
    \end{equation}
    \item The surface gradient of~$f$ is 
\begin{equation}
      \gradg f=\cP\nabla f^e\label{eq:gradg}.
\end{equation}

\item The surface gradient of~$\bs v$ is
\begin{equation}
    \gradg \bs v = \cP (\nabla \bs v^e) \cP .\label{eq:gradgv}
\end{equation}.
\item The surface divergence of~$\bs v$ is
\begin{equation}
    \divg \bs v = \operatorname{tr}(\gradg \bs v). \label{eq:divg}
\end{equation}
\item The divergence of~$A$ is
 \begin{equation}
     \divg A = \begin{pmatrix}
         \divg(e_1^T A)\\
         \divg(e_2^T A)\\
         \divg(e_3^T A)
     \end{pmatrix}.  \label{eq:divgA}
 \end{equation}
 \end{enumerate}

 \end{definition}
It can be easily verified using the product rule that these definitions are independent of the choice of normal extension of~$f,\bs v$, or~$A$. Throughout this work, we shall choose the extension of $\bs n$ to be constant along normal lines.

In this paper we consider the surface Stokes equations
\begin{equation}\label{eqn:surf_stoke}
    \begin{cases}
        -\frac12 \cP \divg \lp\gradg \bu + (\gradg \bu)^T\rp + \gradg p = \cP {\bs f},\\
        \divg \bu = g,\\
                \bu \cdot \bn = 0.
    \end{cases}
\end{equation}
 The quantity~$E[\bu] =- \frac12 \lp \gradg \bu + (\gradg \bu)^T\rp$ in the above expression is frequently referred to as the \emph{stress tensor}. Since~$g$ is the divergence of a vector field on a closed surface, it must be mean-zero. Typically we look for divergence-free~$\bu$ (i.e. $g=0$). We also fix the null space of constant pressures by requiring that the pressure is mean-zero.

\section{Reduction to an integral equation}\label{sec:int_rep}
To derive an integral equation formulation for the surface Stokes equations (\ref{eqn:surf_stoke}), we introduce the 2D Stokeslet and its associated pressure
\begin{equation}\label{eq:k1}
    G(\bx,\by) = \cP \lp -\log r I + \frac{\br \br^T}{r^2} \rp \quad\text{and}\quad P_G(\bx,\by) = \frac{\br^T}{r^2} ,
\end{equation}
where~$\br$ is treated as a column vector. We also introduce the kernel
\begin{equation}\label{eq:k2}
    K(\bx,\by) = \frac{1}{2\pi}\gradg \log r = \cP\frac{\br }{2\pi r^2}.
\end{equation}

Using these kernels, we seek to represent the solution~$\bu$ to equations (\ref{eqn:surf_stoke}) via the ansatz
\begin{multline}\label{eq:ansatz}
    \bu(\br) = \mathcal{G}[\bs \sigma](\br) + \cK[\mu](\br):=\int_{\Gamma} G(\br,\br)\bs \sigma(\br'){\rm{d}}(\br')  + \int_{\Gamma} K(\br,\br)\mu(\br'){\rm{d}}(\br') \\
    \text{and}\quad p(\br) = \cP_G[\bs\sigma](\br)+\mu(\br):=\int_{\Gamma} P_G(\br,\br)\bs \sigma(\br'){\rm{d}}(\br')+\mu(\br),
\end{multline}
where $\bs \sigma$ is an unknown tangential vector field, $\mu$ is an unknown scalar function on $\Gamma.$ We note that this representation automatically satisfies the tangency condition~$\bu\cdot\bn=0$ for any densities~$\bs\sigma$ and~$\mu$. In order to find an equation for $(\bs\sigma,\mu)$ that ensures that~$(\bu,p)$ satisfies~\eqref{eqn:surf_stoke}, we require expressions for the derivatives of our ansatz, which are given in the following subsection.

\subsection{Derivative identities}
In this subsection we summarize several relevant derivative identities related to the kernels defined in (\ref{eq:k1}) and (\ref{eq:k2}). In general, the proofs follow from straightforward, if tedious, application of the formulas in Definition \ref{def:1} and, for ease of exposition, are omitted.
\begin{lemma}
Let ${\bs v}:\Gamma \to \mathbb{R}^3$ be twice continuously differentiable and let ${\bs v}^e$ be its smooth extension to a neighborhood of $\Gamma.$ Then
    \begin{align*}
\nabla_\Gamma \cdot &\nabla_\Gamma \cP {\bs v} = \cP (\Delta - \partial_n^2) {\bs v} - 2 S^T (\nabla v)^T n - ((\Delta - \partial_n^2){\bs n}) {\bs n}\cdot {\bs v} - SS^T{\bs v} -2 H \cP({\bs n} \cdot \nabla) {\bs v},
\end{align*}
where $(\nabla {\bs v})_{i,j} = \partial_j {\bs v}_i,$ $S_{i,j} := \partial_j n_i$ is the \emph{shape operator} (see \cite{demlow2025taylor}) and $H = \frac{1}{2}\partial_j n_j$ is the \emph{mean curvature.} Moreover,
\begin{align*}
    \nabla_\Gamma \cdot &(\nabla_\Gamma \cP {\bs v})^T =  \cP\nabla (\nabla \cdot \mathcal{P}_f {\bs v}) - 2H \cP (\nabla v)^T n - S({\bs n} \cdot \nabla){\bs v}- \cP\nabla(\nabla \cdot \mathcal{P}_f {\bs n}) {\bs n} \cdot {\bs v}\\
    & - S^TS^T {\bs v} - S^T (\nabla v)^T {\bs n}
\end{align*}
where $\mathcal{P}_f$ denotes the `frozen' projection matrix, i.e. its derivatives are taken to be zero.
\end{lemma}

The following corollaries are an immediate consequence of the above lemma.
\begin{corollary}
Let $\bs v = \frac{\bs r \bs r^T}{r^2}\bs w$ for some constant vector $\bs w \in \mathbb{R}^3.$ Then
\begin{align*}
    \cP\nabla_\Gamma \cdot \nabla_\Gamma \cP {\bs v} &= \frac{2}{r^2} \cP{\bs w}-\frac{4}{r^4}({\bs r}\cdot{\bs w}) \cP{\bs r}+\frac{4}{r^4}({\bs n}\cdot{\bs w})({\bs r}\cdot{\bs n})\cP{\bs r}-\frac{8}{r^6} ({\bs r}\cdot{\bs n})^2({\bs r}\cdot{\bs w})\cP{\bs r}\\
    &-2S \frac{{\bs n}\cdot{\bs r}}{r^2} \cP{\bs w} +\frac{4({\bs n}\cdot{\bs r})({\bs r}\cdot{\bs w})}{r^4}S {\bs r} - \cP((\Delta - \partial_n^2){\bs n}) \frac{({\bs n}\cdot{\bs r})({\bs r}\cdot {\bs w})}{r^2} - \frac{{\bs r}\cdot{\bs w}}{r^2}S^2{\bs r}\\
     &-\frac{2({\bs n}\cdot{\bs v})H}{r^2}\cP{\bs r}+\frac{4({\bs r}\cdot{\bs w})({\bs r}\cdot{\bs n})H}{r^4}\cP{\bs r} \numberthis
\end{align*}
and
\begin{align*}
     \nabla_\Gamma \cdot &(\nabla_\Gamma \cP {\bs v})^T =  -\frac{2({\bs r}\cdot{\bs n})H}{r^2}\cP{\bs w}+\frac{4({\bs n}\cdot{\bs r})({\bs r}\cdot{\bs w})H}{r^4}\cP{\bs w}\\
    &+ \frac{1}{r^2}\cP{\bs w}-\frac{2({\bs r}\cdot{\bs w})}{r^4} \cP{\bs r}+\frac{2({\bs n}\cdot{\bs r})({\bs n}\cdot{\bs w})}{r^4}\cP {\bs r}+\frac{2({\bs r}\cdot{\bs n})^2}{r^4}\cP{\bs w}- \frac{8({\bs r}\cdot{\bs n})^2({\bs r}\cdot{\bs w})}{r^6}\cP{\bs r}\\
    &-\frac{{\bs n}\cdot{\bs w}}{r^2}Sr+\frac{2({\bs r}\cdot{\bs n})({\bs r}\cdot{\bs w})}{r^4}S {\bs r} -\frac{({\bs n}\cdot{\bs r})({\bs r}\cdot{\bs w})}{r^2}\cP\nabla(\nabla \cdot \mathcal{P}_f {\bs n}) -\frac{{\bs r}\cdot{\bs w}}{r^2}S^2 {\bs r}\\
    & -S \frac{{\bs n}\cdot{\bs r}}{r^2} \cP{\bs w} +\frac{2({\bs n}\cdot{\bs r})({\bs r}\cdot{\bs w})}{r^4}S {\bs r}.\numberthis
\end{align*}
Finally, we also have
\begin{align*}
  &\nabla_\Gamma \cdot(\nabla \cP \log(r){\bs w} + (\nabla_\Gamma \log(r) \cP{\bs w})^T) = \\
  &\quad 4 \pi \delta(r) \cP{\bs w} + 2 \frac{(\bs n\cdot \bs r)^2}{r^4} \cP{\bs w} - 3S{\bs r} \frac{{\bs w}\cdot \bs n}{r^2} - 2\log(r) (\bs n \cdot \bs w) \cP \Delta_\Gamma^{\rm comp}{\bs n} - 2 \log(r) S^2 {\bs w} \\
  &\quad - 2H \frac{n\cdot r}{r^2}\cP{\bs w}+\frac{1}{r^2}\cP{\bs w} - 2 \frac{r^T\cP{\bs w}}{r^4} \cP {\bs r}-2H\cP {\bs r} \frac{{\bs w}\cdot \bs n}{r^2}-\frac{\bs n \cdot \bs r}{r^2}S{\bs w},\numberthis
\end{align*}
where $\Delta_\Gamma^{\rm comp}$ denotes the componentwise Laplace-Beltrami operator.
\end{corollary}

The following lemma similarly follows straightforwardly from the formulas in Definition \ref{def:1}.
\begin{lemma}
  Let $\bs v = \frac{\bs r^T}{r^2} \bs w$ for some constant vector $\bs w \in \mathbb{R}^3.$ Then
\begin{align}
    \nabla_\Gamma \bs v &= \frac{1}{r^2}\cP\bs w -2\frac{\cP{\bs r}{\bs r}^T}{r^4} \bs w. 
\end{align}  
\end{lemma}

We summarize these results in the following theorem.
\begin{theorem}
    Let $$L[\br,\bs w]:=\cP \nabla_\Gamma \cdot E\left[-\log(r)\cP\bs w+\cP\frac{rr^T}{r^2}\bs w\right] + \nabla_\Gamma \frac{r^T}{r^2}\bs w$$
    where $\bs w \in\mathbb{R}^3$ is a constant vector. Then
    \begin{align*}
    &L = 2\pi \delta(r) \cP{\bs w} -\frac{2}{r^4}({\bs n}\cdot{\bs w})({\bs r}\cdot{\bs n})\cP{\bs r}+\frac{4}{r^6} ({\bs r}\cdot{\bs n})^2({\bs r}\cdot{\bs w})\cP{\bs r} +S \frac{{\bs n}\cdot{\bs r}}{r^2} \cP{\bs w} -\frac{2({\bs n}\cdot{\bs r})({\bs r}\cdot{\bs w})}{r^4}S {\bs r}\\
  & +  \frac{(\bn\cdot \br)^2}{r^4} \cP{\bs w} - \frac{3}{2}S{\bs r} \frac{{\bs w}\cdot \bn}{r^2} - \log(r) (\bn \cdot {\bs w}) \cP \Delta_\Gamma^{\rm comp}{\bs n} -  \log(r) S^2 {\bs w} - H \frac{\bn\cdot \br}{r^2}\cP{\bs w}\\
  &-HP {\bs r} \frac{{\bs w}\cdot \bn}{r^2}-\frac{\bn \cdot \br}{2r^2}S{\bs w}+\frac{({\bs r}\cdot{\bs n})H}{r^2}\cP{\bs w}-\frac{2({\bs n}\cdot{\bs r})({\bs r}\cdot{\bs w})H}{r^4}\cP{\bs w}\\
   & +\frac{1}{2} \cP((\Delta - \partial_n^2){\bs n}) \frac{({\bs n}\cdot{\bs r})({\bs r}\cdot {\bs w})}{r^2}
    +\frac{1}{2}\frac{{\bs r}\cdot{\bs w}}{r^2}S^2{\bs r}
    +\frac{({\bs n}\cdot{\bs w})H}{r^2}\cP{\bs r}-\frac{2({\bs r}\cdot{\bs w})({\bs r}\cdot{\bs n})H}{r^4}\cP{\bs r} \\
    &-\frac{({\bs n}\cdot{\bs r})({\bs n}\cdot{\bs w})}{r^4}\cP {\bs r}-\frac{({\bs r}\cdot{\bs n})^2}{r^4}\cP{\bs w}+ \frac{4({\bs r}\cdot{\bs n})^2({\bs r}\cdot{\bs w})}{r^6}\cP{\bs r}+\frac{{\bs n}\cdot{\bs w}}{2r^2}Sr-\frac{({\bs r}\cdot{\bs n})({\bs r}\cdot{\bs w})}{r^4}S {\bs r}\\
     &+\frac{({\bs n}\cdot{\bs r})({\bs r}\cdot{\bs w})}{2r^2}\cP\nabla(\nabla \cdot \mathcal{P}_f {\bs n})+\frac{{\bs r}\cdot\bs w}{2r^2}S^2 {\bs r}+S \frac{{\bs n}\cdot{\bs r}}{2r^2} \cP\bs w -\frac{({\bs n}\cdot{\bs r})({\bs r}\cdot \bs w)}{r^4}S {\bs r}+ \frac{(\bn\cdot \br)(\bn\cdot \bs w)}{r^4}\cP r.
\end{align*}
\end{theorem}
For notational convenience, we define the following function.
\begin{definition}\label{def:2}
    With $L[\br, \bs w]$ defined as in the previous theorem, let
    $$K_{G,1}(\br,\br')[\bs w] := L[\br-\br',\cP(\br')\bs w]- 2\pi \delta (\br-\br') \cP(\br') \bs w.$$
    In particular, $K_{G,1}(\br,\br')$ is the matrix-valued function obtained by excluding the delta function from $L,$ and projecting $\bs w$ to its tangential components.
\end{definition}

The following lemma gives an expression for $G$ when it is substituted in for $\bs u$ in the left-hand side of the second equation of (\ref{eqn:surf_stoke}).
\begin{lemma}
Let
$${\bs v}=  \left(-\log(r)I+ \frac{{\bs r}{{\bs r}^T}}{r^2} \right){\bs w}.$$
Then
\begin{align}
    \nabla_\Gamma \cdot \cP {\bs v}&=2\left[\frac{({\bs n}\cdot {\bs r})^2({\bs r}\cdot {\bs w})}{r^4} + H \log(r)\,({\bs n}\cdot {\bs w}) -H\frac{({\bs n}\cdot{\bs r})({\bs r}\cdot{\bs w})}{r^2} \right].
\end{align}
\end{lemma}
For notational convenience, we define the following function.
\begin{definition}\label{def:3}
    Let $k_{G,2}$ be the function defined by
    \begin{align}
    k_{G,2}(\br)[\bs w] :=2 \left[\frac{({\bs n}\cdot {\bs r})^2({\bs r}\cdot {\bs w})}{r^4} + H \log(r)\,({\bs n}\cdot {\bs w}) -H\frac{({\bs n}\cdot{\bs r})({\bs r}\cdot{\bs w})}{r^2} \right]
    \end{align}
    where $\bs w \in \mathbb{R}^3,$
    and $K_{G,2}$ be defined by
    \begin{align}
        K_{G,2}(\br,\br')[\bs w]:= k_{G,2}(\br-\br') \cP(\br')\bs w.
    \end{align}
\end{definition}
The following lemma gives the derivative of $K$ when substituted into the first equation.
\begin{lemma}
    Let $\bs v = \nabla_\Gamma \log r.$ Then
    \begin{align*}
    &\cP \nabla_\Gamma \cdot (\nabla_\Gamma \bs v + (\nabla_\Gamma \bs v)^T) = 2 \pi \cP \nabla \delta(\bs r)\\
     &\quad -8 \frac{(\bs n\cdot \bs r)^2}{r^6} \cP {\bs r}+ \frac{4({\bs r}\cdot {\bs n})}{r^4} S{\bs r} -[(\Delta - \partial_n^2){\bs n}] \frac{{\bs n}\cdot{\bs r}}{r^2} -\frac{1}{r^2}S^2{\bs r}+ \frac{4H({\bs n}\cdot{\bs r})}{r^4} \cP{\bs r}\\
     &\quad+\frac{4H({\bs r}\cdot {\bs n})}{r^4}\cP{\bs r}-\frac{8({\bs n}\cdot{\bs r})^2}{r^6}\cP{\bs r}+\frac{2({\bs n}\cdot {\bs r})}{r^4}S{\bs r} -\cP\nabla(\nabla \cdot \mathcal{P}_f{\bs n}) \frac{({\bs n}\cdot {\bs r})}{r^2}\\
    &\quad  -\frac{1}{r^2}S^2{\bs r}+ \frac{2({\bs r}\cdot {\bs n})}{r^4}S{\bs r}.\numberthis
    \end{align*}
\end{lemma}
The previous lemma motivates the following definition.
\begin{definition}\label{def:4}
    Let $k_{K,1}$ be the vector-valued function defined by
    \begin{align*}
        2 \pi k_{K,1}(\bs r) &= -8 \frac{(n\cdot r)^2}{r^6}\cP {\bs r}+ \frac{4({\bs r}\cdot {\bs n})}{r^4} S{\bs r} -[(\Delta - \partial_n^2){\bs n}] \frac{{\bs n}\cdot{\bs r}}{r^2} -\frac{1}{r^2}S^2{\bs r}+ \frac{4H({\bs n}\cdot{\bs r})}{r^4} \cP{\bs r}\\
     &\quad+\frac{4H({\bs r}\cdot {\bs n})}{r^4}\cP{\bs r}-\frac{8({\bs n}\cdot{\bs r})^2}{r^6}\cP{\bs r}+\frac{2({\bs n}\cdot {\bs r})}{r^4}S{\bs r} -\cP\nabla(\nabla \cdot \mathcal{P}_f{\bs n}) \frac{({\bs n}\cdot {\bs r})}{r^2}\\
    &\quad  -\frac{1}{r^2}S^2{\bs r}+ \frac{2({\bs r}\cdot {\bs n})}{r^4}S{\bs r}.\numberthis
    \end{align*}
    and
    \begin{align}
        K_{K,1}(\br,\br'):= k_{K,1}(\br-\br').
    \end{align}
\end{definition}

We conclude with the following lemma and corresponding definition.
\begin{lemma}
 Let ${\bs v} = P\frac{\br}{2\pi r^2}.$ Then
 \begin{align}
     \nabla_\Gamma \cdot \bs v = \delta(\bs r) + \frac{2 (\bs n \cdot \bs r)^2}{2 \pi r^4}-2 H \frac{\bs n \cdot \bs r}{2 \pi r^2}.
 \end{align}
\end{lemma}
Associated with the above identity, we have the following definition.
\begin{definition}\label{def:5}
    Let $k_{K,2}$ be the function defined by
    \begin{align}
        k_{K,2}(\bs r) = \frac{2 (\bs n \cdot \bs r)^2}{2 \pi r^4}-2 H \frac{\bs n \cdot \bs r}{2 \pi r^2}.
    \end{align}
    and $K_{K,2}$ be defined by
    \begin{align}
        K_{K,2}(\br,\br') = k_{K,2}(\br-\br').
    \end{align}
\end{definition}
\subsection{The integral equation and its properties}
Collecting the identities given in the previous section, we see that if the ansatz (\ref{eq:ansatz}) is substituted into the first and second equations of (\ref{eqn:surf_stoke}) then we obtain the following system of integral equations for the unknown densities $\bs \sigma$ and $\mu,$
\begin{align}
2\pi \bs \sigma + \mathcal{K}_{G,1}[\bs \sigma]+\mathcal{K}_{K,1}[\mu] &=\cP \bs f \label{eq:inteq1}\\
\mu + \mathcal{K}_{G,2}[\bs \sigma] + \mathcal{K}_{K,2}[\mu]&=g ,\label{eq:inteq2}
\end{align}
where $\mathcal{K}_{G,1},\mathcal{K}_{G,2},\mathcal{K}_{K,1},$ and $\mathcal{K}_{K,2}$ are the integral operators with kernels $K_{G,1},K_{G,2},K_{K,1},$ and $K_{K,2}$ given in Definitions \ref{def:2}, \ref{def:3}, \ref{def:4}, and \ref{def:5}, respectively (defined analogously to \eqref{eq:ansatz}). 

\begin{remark}
    For notational convenience, and ease of exposition, we have written the above equation for $\bs \sigma : \Gamma \to \mathbb{R}^3,$ rather than explicitly parameterizing it as a tangential vector field. The fact that a $\bs \sigma$ obtained from solving these equations is indeed tangential can be deduced by taking the inner product of (\ref{eq:inteq1}) with $\bn$ and observing that $\bn \cdot \mathcal{K}_{G,1} =0 = \bn \cdot \mathcal{K}_{K,1}.$ 
\end{remark}

The following theorem establishes important properties of the operators appearing in the integral equation system (\ref{eq:inteq1},\ref{eq:inteq2}). 

\begin{theorem}
Let $\mathcal{K}_{G,1},\mathcal{K}_{G,2},\mathcal{K}_{K,1},$ and $\mathcal{K}_{K,2}$ be the integral operators with kernels $K_{G,1},K_{G,2},K_{K,1},$ and $K_{K,2}$ given in Definitions \ref{def:2}, \ref{def:3}, \ref{def:4}, and \ref{def:5}, respectively. If $\Gamma$ is a smooth and compact manifold without boundary then for any $p \in \mathbb{R}_{\ge 0}$ and any $\epsilon \in(0,1],$
\begin{enumerate}
    \item $\mathcal{K}_{G,1}:H^{p}(\Gamma;\mathbb{R}^3) \to H^{p+2-\epsilon}(\Gamma;\mathbb{R}^3)$
    \item $\mathcal{K}_{G,2}:H^{p}(\Gamma;\mathbb{R}^3) \to H^{p+2-\epsilon}(\Gamma;\mathbb{R})$
    \item $\mathcal{K}_{K,1}:H^p(\Gamma;\mathbb{R}) \to H^{p+1-\epsilon}(\Gamma;\mathbb{R}^3),$ and
    \item $\mathcal{K}_{K,2}: H^{p}(\Gamma;\mathbb{R}) \to H^{p+2-\epsilon}(\Gamma;\mathbb{R}),$
\end{enumerate}
where $H^p$ denotes the standard Sobolev space on $\Gamma$.
Moreover, for any $p \ge 0,$ as operators from $H^p \to H^p,$ all of the above operators are compact.
\end{theorem}

\begin{proof}
We prove this result by identifying the singularities of the kernels when~$\br\approx \br'$. We begin by identifying, and Taylor expanding, a few prototypical singular terms. For a smooth surface $\Gamma,$ if $\br:\mathbb{R}^2 \to \Gamma$ is a local coordinate system then
\begin{align*}
\br (u,v) &- \br (u',v') = \br_u(u',v')(u-u') + \br_v(u'v') (u',v') (v-v') + \frac{1}{2} \br_{uu}(u',v') (u-u')^2 +  \\
&+\br_{uv}(u',v') (u-u')(v-v')+ \frac{1}{2}\br_{vv}(u',v')(v-v')^2+O((u-u')^3+(v-v')^3),
\end{align*}
where subscripts denote partial derivatives. It follows immediately that
\begin{align}
(\bs r(u,v) - \bs r'(u',v')) \cdot \bs n (u',v') = \mathrm{I\!I}(\bu'-\bv') +O((u-u')^3+(v-v')^3),
\end{align}
where $\bu = (u,v),$ $\bu' = (u',v'),$ and $\mathrm{I\!I}$ denotes the second fundamental form at $(u',v').$ Similarly,
$$\|\bs r(u,v) - \bs r(u',v')\|^2 = \mathrm{I}(\bu-\bu')+O((u-u')^3+(v-v')^3),$$
where $\mathrm{I}$ denotes the first fundamental form. In particular, it follows immediately (see \cite{hsiao2021boundary} for example) that
\begin{align}
F(u,v;u',v'):=\frac{(\br(u,v)- \br(u',v'))\cdot \bn (u',v')}   {\|\br(u,v)- \br(u',v')\|^2} = \kappa_{u',v'}(\bu-\bu') +O(\|\bu-\bu'\|)
\end{align}
where $\kappa_{u',v'}$ denotes the normal curvature of $\Gamma$ at the point $\br(u',v').$ We further note that the $O(\|\bu-\bu'\|)$ remainder is differentiable in $u$ and $v$ except at $(u',v')$, and the derivatives are bounded in the vicinity of $(u',v').$ It is similarly easy to show that $\nabla_{(u,v)} F$ and $\nabla_{(u',v')} F$ are both of the form $\frac{G_F(u,v;u',v')}{\|\bu-\bu'\|}$ where $G_F$ is bounded, and smooth away from the diagonal. If we set $F^*(u,v;u',v') = F(u',v';v,u)$ then similar arguments show that $F^*$ satisfies all of the same estimates as $F.$

Next, we consider
\begin{align*}
R(u,v;u',v'):= \frac{1}{\|\bs r(u,v) - \bs r (u',v')\|} \cP(u',v') \bs n(u,v),
\end{align*}
where $P(u',v')$ is the $3\times 3$ projection matrix onto the tangent plane at $(u',v').$ Performing a Taylor expansion on $\bs n,$ we find
\begin{align*}
    \bs n(u,v) &= \bs n(u',v') + S \bs r_u(u',v') (u-u') + S \bs r_v(u',v') (v-v') + O(\|\bu -\bu'\|^2),\\
    &= \bs n(u',v') + S (\br(u,v)-\br(u',v')) + O(\|\bu-\bu'\|^2),
\end{align*}
where $S$ is the shape operator of $\Gamma$ at $\br (u',v').$ Thus,
\begin{align*}
    R(u,v;u',v') = S \frac{\br(u,v)-\br(u',v')}{\|\br(u,v)-\br(u',v')\|}+O(\|\br(u,v)-\br(u',v')\|). 
\end{align*}
As for $F,$ we note that the remainder is differentiable in $u$ and $v$ except at $(u',v')$ with bounded derivatives, at least in the vicinity of $(u',v').$ It is also clear that $\nabla_{(u,v)} R$ and $\nabla_{(u',v')} R$ are both of the form $\frac{G_R(u,v;u',v')}{\|\bu-\bu'\|}$ where $G_R$ is bounded, and smooth away from the diagonal. If we set $R^*(u,v;u',v') = R(u',v';v,u)$ then similar arguments show that $F^*$ satisfies all of the same estimates as $R.$

\iffalse
Next, we consider
\begin{align*}
R(u,v;u',v';\alpha,\beta) = \frac{(\alpha(u,v)\,\br_u(u,v)+\beta(u,v)\, \br_v(u,v)) \cdot \bn(u',v')}{\|\br(u,v) - \br(u',v')\|},
\end{align*}
where $\alpha$ and $\beta$ are arbitrary smooth functions. Proceeding as before it is clear that
\begin{align*}
R(u,v;u',v';\alpha,\beta) = \frac{\alpha(u',v') [(u-u')L+(v-v')M] + \beta(u',v')[(u-u')M + (v-v')N] }{\mathrm{I}(\bu-\bu')^{1/2}}+O(\|\bu-\bu'\|),
\end{align*}
where $L,M,$ and $N$ are the coefficients of the second fundamental form at $(u',v'),$ i.e. the matrix of $\mathrm{I\!I}$ at $(u',v')$ is $\begin{bmatrix} L &M\\M&N\end{bmatrix}.$ We note that once again, the remainder is continuous and has a bounded derivative except on the diagonal $(u,v) = (u',v').$ Differentiating, we see that $\nabla_{(u,v)} R$ and $\nabla_{(u',v')} R$ are both of the form $\frac{G_R(u,v;u',v')}{\|\bu-\bu'\|}$ where $G_R$ is bounded, and smooth away from the diagonal. We remark that 
\fi 

Finally, we observe that if
$$T(u,v;u',v'):= \frac{(\br(u,v)-\br(u',v'))}{\|\br(u,v)-\br(u',v')\|},$$
then $\nabla_{(u,v)}T$ and $\nabla_{(u,'v')}T$ are both of the form $\frac{G_T(u,v;u',v')}{\|\bu-\bu'\|}$ where $G_H$ is bounded, and smooth away from the diagonal. 

Looking at Definitions \ref{def:2}, \ref{def:3}, \ref{def:5} we can see that all of the terms in $K_{G,1},K_{G,2},$ and $K_{K,2}$ are of the form $\phi(\bu,\bu',F,F^*,R,R^*,T)+\log(r)M$,  where $\phi: \Gamma \times \Gamma \times \mathbb{R}^5\to \mathbb{R}^{j\times k},$ $j,k \in \{1,2,3\},$  and $M \in H^2(\Gamma; \mathbb{R}^{j \times k})$ is a (possibly matrix-valued) function on $\Gamma$.
The derivatives of these kernels have at most a $1/\|\bu - \bu'\|$ singularity near the diagonal, and so standard results on weakly-singular integral operators (e.g. Theorem 4.1 in \cite{punchin1988weakly}) give that the derivatives of the integral operator $\mathcal{K}$ with kernel 
map $H^p(\Gamma;\mathbb{R}^j) \to H^{p+1-\epsilon}(\Gamma;\mathbb{R}^k)$ for any non-negative $p.$  Thus $\mathcal{K}$ itself
maps $H^p(\Gamma; \mathbb{R}^j) \to H^{p+2-\epsilon}(\Gamma; \mathbb{R}^k)$ for any $\epsilon \in (0,1],$ and any $p \in \mathbb{R}_{\ge 0}$ and we have proved the desired mapping properties of $\cK_{G,1},\cK_{G,2},$ and $\cK_{K,2}$ .

The last kernel, $K_{K,1},$ contains terms of the form
\begin{align*}
\phi(\bu,\bu',F,F^*,R,R^*,T)\frac{\br}{r^2},
\end{align*}
in addition to terms of the above form. It follows that $\mathcal{K}_{K,1}:H^p(\Gamma;\mathbb{R}) \to H^{p+1-\epsilon}(\Gamma;\mathbb{R}^3),$ from similar reasoning as above. The compactness of the integral operators then follows from their mapping properties and the Sobolev embedding theorem.
\end{proof}

The following corollary is an immediate but important consequence of the previous theorem.
\begin{corollary}
The system of integral equations (\ref{eq:inteq1},\ref{eq:inteq2}) are Fredholm-second kind. In particular, existence and uniqueness are equivalent, and the space of solutions to the homogeneous problem (with $\cP \bs f = g = 0$) is finite dimensional.
\end{corollary}
A standard bootstrapping argument establishes improved regularity of $\bs \sigma$ and $\mu$ if $\cP \bs f$ and $g$ are $C^k(\Gamma).$ The proof follows immediately from the mapping properties of $\mathcal{K}_{G,1},\cdots,\mathcal{K}_{K,2}.$
\begin{corollary}
Let $\bs \sigma$ and $\mu$ be solutions to (\ref{eq:inteq1},\ref{eq:inteq2}) with right-hand sides $\cP \bs f \in C^k(\Gamma,\mathbb{R}^3)$ and $g \in C^{k+1}(\Gamma,\mathbb{R})$ for $k>1.$ Then $\bs \sigma \in C^k(\Gamma,\mathbb{R}^3)$ and $\mu \in C^{k}(\Gamma,\mathbb{R}).$
\end{corollary}

As noted above, solvability of \eqref{eqn:surf_stoke} requires~$g$ to be mean zero. We also note that the pressure $p$ is only defined up to a constant, and so we are free to impose the additional constraint that~$\mu$ be mean-zero (see \cite{goodwill2024parametrix,Rachh2015}). It is easily seen that the addition of this condition leads to the modified equations:
\begin{align}
2\pi \bs \sigma + \mathcal{K}_{G,1}[\bs \sigma]+\mathcal{K}_{K,1}[\mu] &=\cP \bs f \label{eq:inteq1a}\\
\mu+\frac{1}{|\Gamma|}\int_\Gamma \mu + \mathcal{K}_{G,2}[\bs \sigma] + \mathcal{K}_{K,2}[\mu] &=g \label{eq:inteq2a},
\end{align}
where~$|\Gamma|$ is the area of~$\Gamma$. In particular, the divergence theorem applied to (\ref{eq:inteq2a}) recovers the condition $\int \mu =0.$

\begin{remark}
    It is well-known that if $\Gamma$ has non-trivial Killing vector fields then the PDE (\ref{eqn:surf_stoke}) will have non-trivial homogeneous solutions. The corresponding nullspace of the left-hand side of the system (\ref{eq:inteq1a},\ref{eq:inteq2a}) can be determined and dealt with using standard methods, see for example \cite{Rachh2015}.
\end{remark}

\begin{remark}\label{rem:killing}
    In many papers on computational methods for the surface Stokes equations, a linear term proportional to $\bs u$ is added to the first equation of (\ref{eqn:surf_stoke}) to avoid the necessity of dealing with subtle questions of solvability associated with the Killing fields. The addition of this term is easily treated within our integral equation framework and amounts to a compact perturbation of the system (\ref{eq:inteq1a},\ref{eq:inteq2a}). The resulting system is therefore also Fredholm-second kind.
\end{remark}

 For notational convenience we set $\bs \rho = (\bs \sigma, \mu)^T$ and let $\mathcal{K}_T$ be the integral operator with matrix-valued kernel
$$K_T(\br,\br') = \begin{pmatrix} \frac{1}{2\pi}K_{G,1} & \frac{1}{2\pi}K_{K,1} \\
K_{G,2}(\br,\br') & K_{K,2}(\br ,\br') + \frac{1}{|\Gamma|}\end{pmatrix}.$$
Setting $\bs h = (\frac{1}{2\pi}\cP \bs f , g)^T,$ our integral equations (\ref{eq:inteq1a},\ref{eq:inteq2a}) can be re-written as
\begin{align}\label{eqn:bie_red}
    \bs \rho + \mathcal{K}_T[\bs \rho] = \bs h.
\end{align}

\section{Discretization and Quadrature}\label{sec:disc}
In this section we describe a collocation approach for discretizing the system of integral equations (\ref{eq:inteq1a},\ref{eq:inteq2a}). Though much of it is standard, and indeed our implementation mostly relies on standard tools as implemented in the package \emph{fmm3dbie} \cite{fmm3dbie}, we sketch the details here in the interest of making the description as self-contained as possible.

\subsection{Discretization of the system}
Let $T_0$ denote the standard right-triangle $T_0= \{(x,y) \in \mathbb{R}_{> 0}^2\,|\, x+y < 1\}.$ We assume that the surface $\Gamma$ is given as a set of smooth maps $\psi_i:T_0 \to \Gamma \subset \mathbb{R}^3,$ $i=1,\cdots,\np$ such that the sets $\Gamma_i:=\psi_i(T_0),$ $i=1,\cdots,\np$ are mutually disjoint and $\cup_{i=1}^\np \overline{\psi_i(T_0)} = \Gamma.$ We also assume that the determinants of the metric tensors $g_i$ of the maps $\psi_i$ are bounded uniformly away from zero on the closure of $T_0.$ 

We can then rewrite the integral equation (\ref{eqn:bie_red}) as 
\begin{align}\label{eqn:bie_patchsum}
\bs \rho(\br) + \sum_{i=1}^\np \int_{T_0} K_T(\br, \psi_i(u,v))\, \bs \rho(\psi_i(u,v))\,\sqrt{\det\,g_i(u,v)}\,{\rm d} a(u,v)  = \bs h (\br).
\end{align}
We discretize this system using a collocation method. Suppose for the $k$th patch we fix a basis $\phi_j^k:T_0\to \mathbb{R}^4,$ $j=1,\cdots, M,$ and construct a corresponding $M$ point discretization of $T_0$ with nodes $(u_j^k,v_j^k),$ $j=1,\cdot,M$ which gives stable interpolation of the basis functions $\phi_j^k.$ In particular, we require the matrix $U^k$ with $i,j$th entry $U_{i,j}^k = \phi_j^k(x_i^k),$ $1 \le i,j \le M$ to be well-conditioned. We then proceed by enforcing consistency only at these discretization nodes. In particular, if we set $\bs \rho_i(u,v) = \bs \rho(\psi_i(u,v))$ and $\bs h_i(u,v) = \bs h(\psi_i(u,v))$ then we require
\begin{align}\label{eqn:partial_disc}
\bs \rho_k(u_j^k,v_j^k) + \sum_{i=1}^\np \int_{T_0} K_T(\psi_k(u_j^k,v_j^k), \psi_i(u,v))\, \bs \rho_i(u,v)\,\sqrt{\det\,g_i(u,v)}\,{\rm d}a(u,v)  = \bs h_k(u_j^k,v_j^k). 
\end{align}
To continue, we need a quadrature rule for accurately approximating the integrals on the left-hand side of (\ref{eqn:partial_disc}). Conceptually, if $\bs \rho_i$ is approximately in the span of the basis elements $\phi_i^k,$ $i=1,\cdots,M,$ then 
$$\bs \rho_i(u,v) \approx \sum_{j=1}^M \phi_j^k(u,v) (U^k)^{-1}_{j,k} \bs \rho_i(u_k^i,v_k^i). $$
In particular, if we define the integrals
\begin{align}\label{eqn:a_comp}
    A^{k,i}_{j,\ell}:= \sum_{n=1}^M\left[\int_{T_0} K_T(\psi_k(u_j^k,v_j^k),\psi_i(u,v)) \phi_\ell^n(u,v) \sqrt{\det g_i(u,v)}\,{\rm d}a(u,v)\right] (U^k)^{-1}_{n,\ell}
\end{align}
then
$${\bs \rho}_k(u_j^k,v_j^k) + \sum_{i=1}^\np\sum_{\ell=1}^M A_{j,\ell}^{k,i} {\bs \rho}_i(u_\ell^i,v_\ell^i) \approx \bs h_k(u_j^k,v_j^k).$$
Letting $\hat{\bs \rho}_{j,k}$ denote our unknowns, our discrete system is given by
\begin{align}\label{eq:disc_sys}
 \hat{\bs \rho}_{j,k} + \sum_{i=1}^\np\sum_{\ell=1}^M A_{j,\ell}^{k,i} \hat{\bs \rho}_{\ell,i} \approx \bs h_k(u_j^k,v_j^k).
\end{align}

In our implementation we choose our basis on $T_0$ to be the Koornwinder polynomials of degree at most $q,$ which form an orthogonal basis of the monomials $u^j v^k,$ $j,k \ge 0,$ $j+k \le q.$ As discretization nodes, we choose the corresponding Vioreanu-Rokhlin nodes \cite{vioreanu2014spectra}, see Figure \ref{fig:disc}. With this choice, $M = (q+1)(q+2)/2.$

\begin{figure}
    \centering
    \includegraphics[scale=.15]{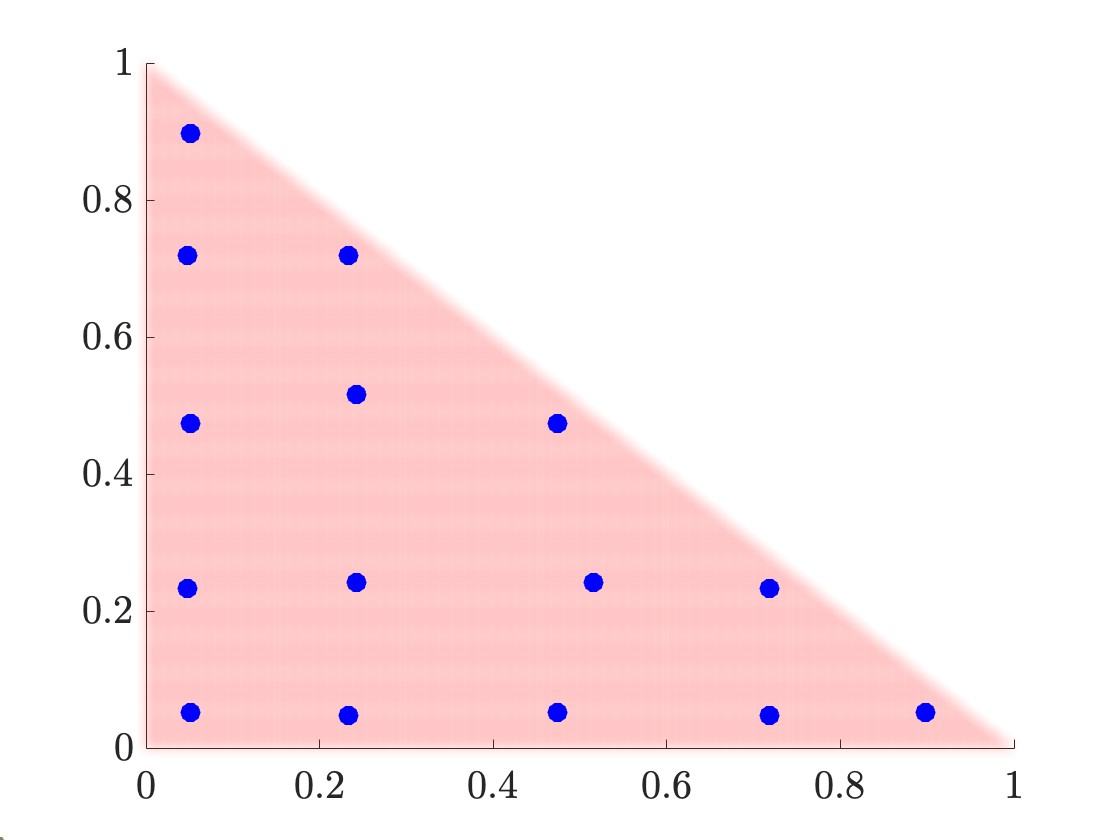}
    \includegraphics[scale=.15]{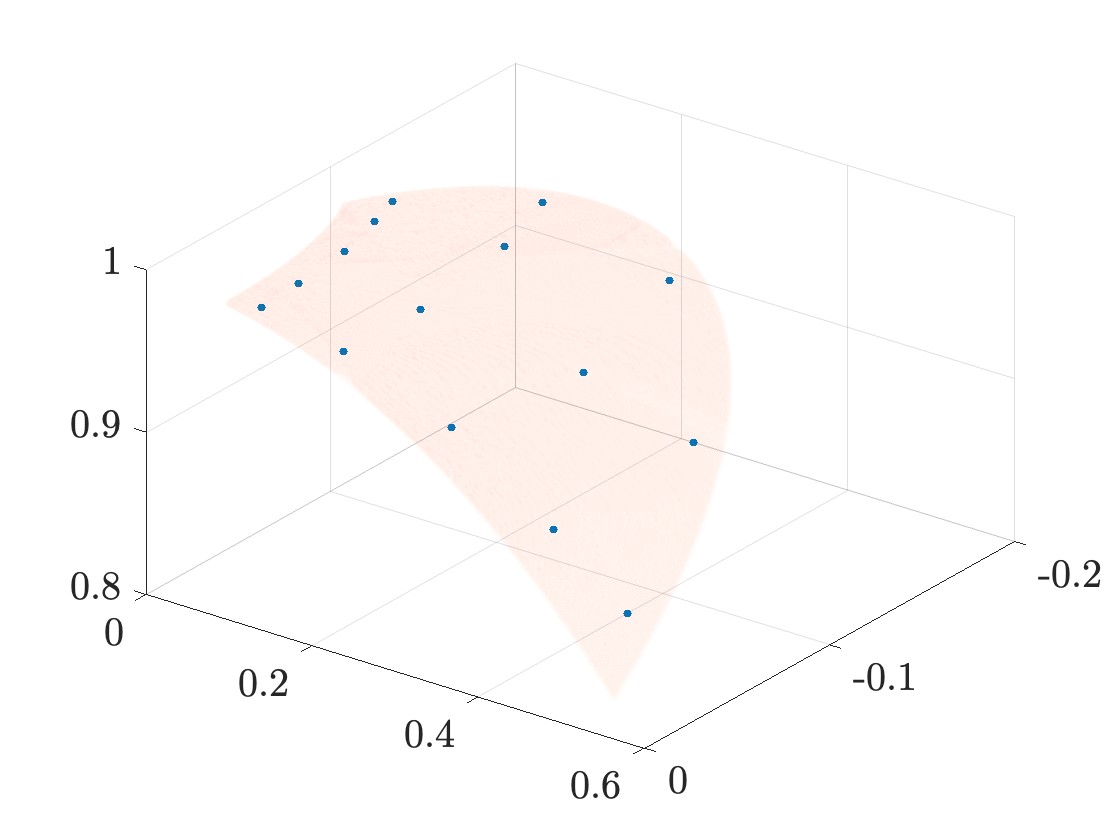}

    \caption{An illustration of the local discretization. Left: the Vioreanu-Rokhlin nodes of order 4 on the standard triangle $T_0.$ The nodes are shown in blue and $T_0$ is shaded red. Right: the image of the Vioreanu-Rokhlin nodes (blue) on a patch $\Gamma_i\subset \mathbb{R}^3$ (red).}
    \label{fig:disc}
\end{figure}

Since we choose the same nodes on each patch, all of the interpolation matrices $U_k$ are the same. We denote this matrix by $U.$ The Vioreanu-Rokhlin nodes are constructed so that the resulting matrix $U$ is well-conditioned. Moreover, associated with the points $(u_j,v_j)$ are corresponding quadrature weights $w_j$ such that the quadrature scheme with nodes $(u_j,v_j)$ and weights $w_j$ integrate all Koornwinder polynomials of degree at most $q$ are integrated exactly.

\subsection{Quadrature for weakly-singular integral operators}
So far we have not addressed how the integrals $A_{j,\ell}^{k,i}$ in (\ref{eqn:a_comp}) are computed numerically. Here we give a brief description of our approach. We refer the interested reader to \cite{loc_cor_quad} for a more complete description of methods of this type. For a fixed `target' patch number $k$ we split the `source' patches $i = 1,\cdots,I$ into three categories: self (i.e. $i=k$), near (i.e. $i \neq k$ but ${\rm dist}(\Gamma_k,\Gamma_i) < \alpha \min \{ {\rm diam} \Gamma_k, {\rm diam} \Gamma_i\}$), and far (all patches which are not self or near). Here the constant $\alpha$ controls a trade-off between speed and accuracy. We again refer the reader to \cite{loc_cor_quad} for a detailed description of the choice of this parameter. In our numerical experiments we use the default settings of \emph{fmm3dbie} \cite{fmm3dbie, bremer, xiaogimbutas}.

\subsubsection{Far quadrature}
When the source patch $\Gamma_i$ is sufficiently separated from the target patch $\Gamma_k,$ the kernel $K_T(\psi_k(u_j^k,v_j^k),\psi_i(u,v))$ is in $C^\infty(\overline{T_0}),$ viewed as a function of $(u,v).$ Similarly, so is the square root of the determinant of the metric tensor $\sqrt{\det g_i(u,v)}.$ Thus, we can use the \emph{native quadrature rule}, i.e.
\begin{align}
    A^{k,i}_{j,\ell}&\approx \sum_{n=1}^M\sum_{m=1}^M\left[K_T(\psi_k(u_j^k,v_j^k),\psi_i(u_m^i,v_m^i)) \phi_\ell^n(u_m^i,v_m^i) \sqrt{\det g_i(u_m^i,v_m^i)}\,w_{m}\right] (U^k)^{-1}_{n,\ell},\\
    &\approx \sum_{m=1}^M\left[K_T(\psi_k(u_j^k,v_j^k),\psi_i(u_m^i,v_m^i)) \sqrt{\det g_i(u_m^i,v_m^i)}\,w_{m}\right] 
\end{align}
where in going from the first line to the second we use the definition of $U^k.$ In the following, we set $\omega_m^i :=\sqrt{\det g_i(u_m^i,v_m^i)} w_m.$

\begin{remark}\label{rem:oversamp}
Depending on the order $q$ of the discretization, it is sometimes convenient to oversample, i.e. use a higher-order Vioreanu-Rokhlin quadrature for the integrals. Though this yields more accurate computation of the desired integrals, the resulting formulas are more involved. In particular, in most applications it is too stringent to assume that one has an analytic form for the surface available at the time when quadratures are being constructed. Instead, the surface and its salient properties (normal derivatives, tangent vectors, curvature, shape operators, etc.) are also represented as Koornwinder polynomial expansions, interpolated from Vioreanu-Rokhlin nodes. See Figure \ref{fig:disc} for an illustration. We once again refer the reader to \cite{loc_cor_quad} for a detailed discussion of this procedure and the corresponding desiderata.
\end{remark}

\subsubsection{Near quadrature}
When the source patch $\Gamma_i$ is close to the target patch $\Gamma_k,$ the kernel $K_T$ is nearly-singular. It follows that it may have extremely slow decay in the coefficients of its expansion in the basis of Koornwinder polynomials. Thus, using either the native quadrature rule or the oversampling procedure in Remark \ref{rem:oversamp} would be either inaccurate or impractical. Instead, we use adaptive integration. In practice, it is typically impractical to assume that we have access to an easily queried analytic parameterization of the surface $\Gamma$ at this point in the calculation. Instead we replace the maps  $\psi_i$ by approximations $\hat{\psi}_i$ given by Koornwinder polynomial expansions. In particular, we store the values of $\psi_i,$ $\partial_u \psi_i,$ $\partial_v \psi_i,$ and the first and second fundamental forms at the Vioreanu-Rokhlin nodes $(u_j,v_j).$ If the derivatives of $\psi_i$ are not given analytically then approximations can be computed via spectral differentiation. Given these approximate maps we are then free to interpolate the relevant quantities to any point in $T_0.$

Our adaptive integration proceeds as follows. First, we evaluate the matrix-entry integrals (\ref{eqn:a_comp}) via the standard native quadrature rule to obtain an approximation $\hat{A}_0$ (for ease of exposition we suppress the dependence on $k,i,j$ and $\ell$). We then cut $T_0$ into four triangles, $T_{0,1},T_{0,2},T_{0,3}$ and $T_{0,4},$ as in Figure \ref{fig:adap}. Next, we use rescaled Vioreanu-Rokhlin nodes and quadrature on each of these subtriangles, which we denote by $\hat{A}_{0,1},$ $\hat{A}_{0,2},$ $\hat{A}_{0,3},$ and $\hat{A}_{0,4},$ respectively. If $|\hat{A}_{0}-\hat{A}_{0,1}-\hat{A}_{0,2}-\hat{A}_{0,3}-\hat{A}_{0,4}| < \epsilon,$ where $\epsilon$ is a user-prescribed tolerance, then we use $\hat{A}_{0,1}+\cdots+\hat{A}_{0,4}$ as our approximation. If the difference is above the tolerance then we subdivide $T_{0,1},\cdots,T_{0,4}$ each into four new triangles, giving $T_{0,i_1,i_2},$ $1 \le i_1,i_2 \le 4,$ and construct corresponding Vioreanu-Rokhlin nodes and weights on each of these subtriangles to compute $\hat{A}_{0,i_1,i_2},$ $1 \le i_1,i_2 \le 4.$ For each $i_1$ we compare $\hat{A}_{0,i_1}$ with $\sum_{i_2=1}^4 \hat{A}_{0,i_1,i_2}.$ If the difference is below the tolerance then we do not subdivide the triangles $\hat{A}_{0,i_1,i_2}$ further. If not, then we once again cut each of the triangles into $4.$ We then proceed recursively, stopping when the sum of the approximate integrals over all the `leaf' triangles (those without any children) are within the tolerance of their parent's integrals. We then set our approximate integrals of $A$ to be the sum computed using the quadrature rules on all of the leaf triangles.

\begin{figure}
    \centering
    \includegraphics[scale=.35]{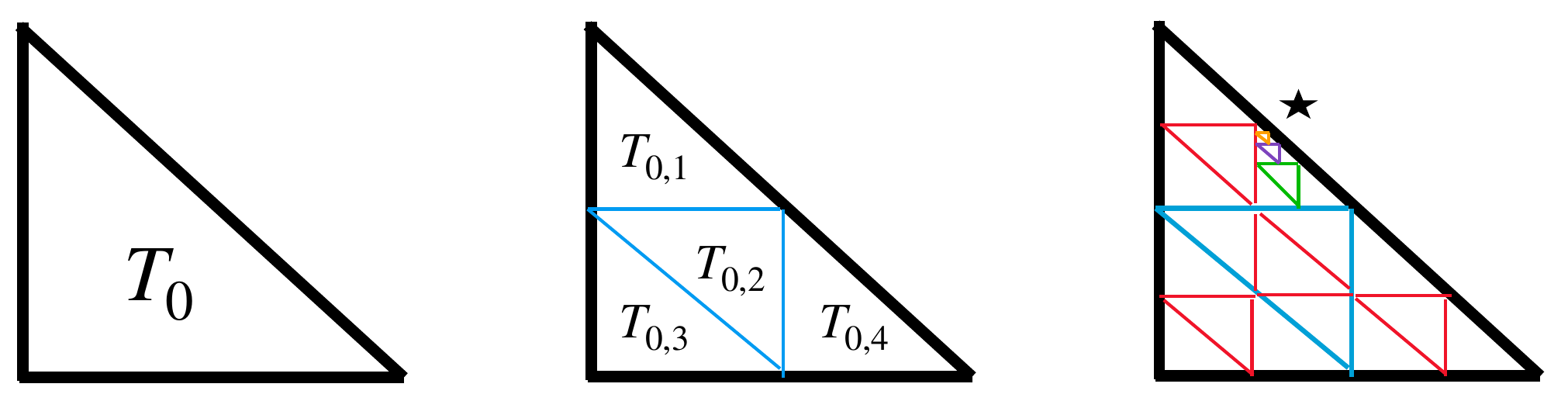}

    \caption{An illustration of the adaptive integration strategy used in the near-field integral computation. Left: the standard triangle $T_0.$ Middle: the first level of refinement. Right: a cartoon of the refinement used for a source located at the point indicated by a star.}
    \label{fig:adap}
\end{figure}

\begin{remark}
    In practice it is better to use Xiao-Gimbutas \cite{xiaogimbutas} quadrature nodes to perform the integration in this setting. The order $q$ Xiao-Gimbutas quadrature rule is designed to integrate more than the first $(q+1)(q+2)/2$ Koornwinder polynomials, and hence is more efficient than the same order Vioreanu-Rokhlin quadrature.
\end{remark}
\begin{remark}
    Other refinement strategies are possible. One common strategy is to refine each triangle until either all its subtriangles are separated from the target by a constant factor times their diameter or the triangles are sufficiently small that their contribution to the integral is less than the desired tolerance.
\end{remark}

\subsubsection{Self quadrature}

When the source patch and target patch coincide, the integrands will, in general, have slow or no decay in the coefficients of their Koornwinder polynomial expansion, and may have singularities in their derivatives as $(u,v) \to (u_j^k,v_j^k).$ In principle, the adaptive integration approach described in the previous subsection may still be used to approximate the integrals. In practice, this can become prohibitively slow, as many levels are required to resolve the singular behavior of the integrand. Instead, a standard approach is to use {\it generalized Gaussian quadratures} \cite{bremer}.

The first step is to apply an affine transformation $(u,v)\mapsto (\tilde{u},\tilde{v})$ to $T_0$ with $(u_j^k,v_j^k) \mapsto (0,0),$ which makes the first fundamental form of $\Gamma_k$ the identity at $\psi_k(u_j^k,v_j^k) = \tilde{\psi}_k(0,0),$ with respect to this new coordinate system. The triangle is then split into three subtriangles, each with a vertex at $(0,0).$ On each of these triangles integrals are performed using appropriate quadratures, constructed via the approach in \cite{bremer}, which accurately integrate singularities of the form: $$\frac{p_1(\tilde{u},\tilde{v)}}{\sqrt{\tilde{u}^2+\tilde{v}^2}} + p_2(\tilde{u},\tilde{v})\log(\tilde{u}^2+\tilde{v}^2) + p_3(\tilde{u},\tilde{v}) T(\cos\phi,\sin \phi) + p_4(\tilde{u},\tilde{v})$$
for arbitrary polynomials $p_1,p_2,$ and $p_3$ of degree at most $p$ (depending on the particular quadrature rule employed) and $T$ is a polynomial in $(\cos \phi,\sin \phi):= (\tilde{u},\tilde{v})/\sqrt{\tilde{u}^2+\tilde{v}^2},$ to within a prefixed tolerance (typically machine precision or better).

\section{Acceleration of the computation}\label{sec:accel}
The Faustian bargain we make when recasting the PDE system (\ref{eqn:surf_stoke}) as a system of integral equations is that we obtain a well-conditioned Fredholm second-kind system, but lose the sparsity that comes with discretizing local operators directly. Indeed, this is the \emph{raison d'\^{et}re} for a number of fast algorithms \cite{martinsson2019fast,martinsson2025fast,loc_cor_quad,ho2016hierarchical,hackbusch1989fast}. The majority of these algorithms are based on low rank factorizations of blocks of the discretized system matrix. In this section, we illustrate a simple algorithm reminiscent of the method presented in Chapter 13 of \cite{martinsson2019fast} for solving~\eqref{eq:disc_sys} that leverages these factorizations. More efficient solvers can be constructed by building hierarchical low rank factorizations (Chapter 14 of~\cite{martinsson2019fast}) and in our numerical examples below we use the method described in~\cite{sushnikova2023fmm}, adapted to handle this vector-valued problem.

Typically, for standard fast algorithms such as \emph{inter alia} the \emph{fast multipole method} \cite{greengard1987fast,tornberg,gumerov2006fast} and \emph{fast direct solvers} \cite{martinsson2005fast,greengard2009fast,sushnikova2023fmm,minden2017recursive,corona2015n}, the method relies on the kernels satisfying a Green's identity. Though there are algorithms that only leverage smoothness of the function (for example\cite{Fong2009,ho2020flam}), the overall computational cost can be significantly larger relative to the classical methods, depending on the form of the kernels and their associated rank structure. Our kernels fall between these two categories. While they do not satisfy Green's identities in three dimensions, they are built out of two dimensional fundamental solutions of PDEs and their derivatives. In this section, we show how to use this fact to analyze the rank structure of our discretized operators. The approach extends naturally to a much broader class of problems. 

As motivation, we briefly sketch a simple compression-based algorithm which exploits a certain low rank structure of the system matrix $A$ for solving systems like (\ref{eq:disc_sys}).  We partition the surface $\Gamma$ into~$n$ boxes, assuming each box has approximately equal number of discretization nodes, and let~$A_{ij}$ denote the sub-block of~$A$ corresponding to interactions between points in box $j$ with points in box~$i$. If we let~$\rho_i$ and~$h_i$ denote the degrees of freedom and data in box~$i$ respectively, then \eqref{eq:disc_sys} is equivalent to the block-system
\begin{equation} \label{eq:block_sys}
    \begin{pmatrix}
        A_{11} & \cdots & A_{1n}\\
        \vdots &&\vdots\\
        A_{n1} & \cdots & A_{nn}
    \end{pmatrix}\begin{pmatrix}
        \rho_1\\ \vdots \\ \rho_n
    \end{pmatrix}= \begin{pmatrix}
        h_1\\ \vdots \\ h_n
    \end{pmatrix}.
\end{equation}

In the next section we will show that for each box $B_i$ there exist rectangular matrices~$X_i \in \mathbb{R}^{N_{i}\times N_{s,i}}$ and~$E_i\in \mathbb{R}^{N_{s,i}\times N_{i}},$ $N_{s,i} < N_i$ such that each sub-block $A_{ij}$ with~$i\neq j$ admits the approximate low-rank factorization
\begin{equation}\label{eq:blockfac}
    A_{ij} \approx E_i \hat{A}_{ij} X_j,
\end{equation}
where~$\hat{A}_{ij}$ is a sub-matrix of~$A_{ij}$. Let $\tau = \max_i N_{s,i}/N_i.$ Assuming for now that we have such an approximate factorization, if we substitute it into our system and multiply each row by~$A_{ii}^{-1}$, then~\eqref{eq:block_sys} becomes
\begin{equation}\label{eq:block_skel}
        \begin{pmatrix}
        I &A_{11}^{-1} E_1 \hat{A}_{12} X_2 &  \cdots & A_{11}^{-1} E_1 \hat{A}_{1n} X_n\\
        \vdots && &\vdots\\
        A_{nn}^{-1} E_n \hat{A}_{n1} X_1 & \cdots & A_{nn}^{-1} E_n \hat{A}_{n(n-1)} X_{n-1} & I
    \end{pmatrix}\begin{pmatrix}
        \rho_1\\ \vdots \\ \rho_n
    \end{pmatrix}= \begin{pmatrix}
        A_{11}^{-1}h_1\\ \vdots \\ A_{nn}^{-1}h_n
    \end{pmatrix}.
\end{equation}
To continue, we introduce the \emph{skeleton variables}~$\hat \rho_i = X_i \rho_i$ for all~$i$. If we solve for each~$\hat\rho_i$, then we can use the above equations to easily compute all the full~$\rho_i$. If N denotes the total number of points,
$\tilde N=\max N_i$ and~$\tilde N_s = \max N_{s,i}$, then this step will take
$O( n\tilde N^3 + n^2 \tilde N \tilde N_s)= O( N\tilde N^2 + N^2 \tau )$ time. We are thus free to multiply each row of~\eqref{eq:block_skel} by~$X_i$, giving rows of the form
\begin{equation}
    \hat\rho_i +\sum_{j\neq i}X_i A_{ii}^{-1} E_i\hat{A}_{ij}\hat\rho_j= X_iA_{ii}^{-1}h_i.
\end{equation}
If we let~$S_{ii} = X_i A_{ii}^{-1} E_i$ and multiply each row by~$S_{ii}^{-1}$, we arrive at the reduced linear system
\begin{equation}
        \begin{pmatrix}
        S_{11}^{-1} & \hat{A}_{12}  &  \cdots &\hat{A}_{1n} \\
        \vdots && &\vdots\\
         \hat{A}_{n1} & \cdots &\hat{A}_{n(n-1)} & S_{nn}^{-1}
    \end{pmatrix}\begin{pmatrix}
       \hat\rho_1\\ \vdots \\\hat\rho_n
    \end{pmatrix}= \begin{pmatrix}
        S_{11}^{-1}X_1A_{11}^{-1}h_1\\ \vdots \\ S_{nn}^{-1}X_nA_{nn}^{-1}h_n
    \end{pmatrix}.
\end{equation}
Due to the restriction that~$\hat{A}_{ij} $ is a sub-matrix of~$A$, this reduced matrix has the same structure as~$A$. The above compression scheme can therefore be repeated with larger boxes~$\tilde B_i$ until the reduced system is small enough to solve direction. This idea forms the basis of fast algorithms such as the recursive skeletonization algorithm~\cite{ho2020flam}. There also exist more sophisticated approaches such as~\cite{sushnikova2023fmm,minden2017recursive}, which relax the requirement that~\eqref{eq:blockfac} holds for all~$i\neq j$, which reduces the size of the low rank factorizations, and so achieves an $O(NM)$ scaling algorithm. In our experiments below, we use an extension of the method presented in \cite{sushnikova2023fmm}, allowing us to build a fast direct solver for~\eqref{eq:disc_sys} in $O(NM)$ time, albeit with a larger constant.

In the above, for ease of exposition we have assumed that the matrices $S_{ii}$ are invertible. This is not required in actual implementations. We stress that in order for an algorithm to beat the naive matrix-vector multiplication cost $O(N^2),$ the compression step~\eqref{eq:blockfac} must be done without computing all of the matrix entries of $A.$ It is here where the particular structure and analytical properties of the kernels must be leveraged. In the following section we discuss a method for efficiently constructing these factorizations.

\subsection{Construction of the low rank factorizations}\label{sec:build_compress}
We now discuss how to build the factorizations~\eqref{eq:block_sys} without needing to form the entire matrix $A$. In this section we describe how to construct the~$X_j$'s. The intuition for the $E_i$'s is analogous, with a few modifications described in the next section. The method described in this section is based on an extension of the \emph{proxy surface} method \cite{ye2020analytical,xing2020interpolative}, which was used to build the fast direct solver in~\cite{sushnikova2023fmm}. Our extension is a special case of the work in \cite{xing2020interpolative} which applies to general kernels, but the structure of our kernels will allow us to prove that our ranks do not grow as the box sizes do and to give more explicit error bounds.

When constructing the boxes~$B_i$, we can choose centers~$\bx_i$ and ensure that each box~$B_i$ lives in the ball~$B_\rho(\bx_i)$. 
When~$\|\bx_i-\bx_j\|>(1+c)\rho$, these blocks are said to be ``well-separated". We will choose~$\rho$ large enough so that the matrix entries corresponding to points in well-separated boxes are evaluated using the farfield quadrature. For simplicity we assume that this is obtained by taking the kernels evaluated at source and target points multiplied by a smooth quadrature weight depending only on the source.

We begin by describing a method for efficiently constructing an~$X_j$ such that
\begin{equation}
    A_{ij} \approx \tilde A_{ij} X_j,
\end{equation}
where~$\tilde A_{ij}$ is a sub-matrix of~$A_{ij}$, for all~$i$ such that~$B_i$ and~$B_j$ are well separated. Throughout this section, we let~$\br'_1,\ldots,\br'_N$ be the discretization nodes in~$B_j$ and~$\br_1,\ldots$ are the rest of the discretization nodes.

We begin with the following lemma, which is an easy extension of the observation in \cite{tornberg}.
\begin{lemma}
Let $r = \sqrt{(x_1-y_1)^2+(x_2-y_2)^2}.$ Then
 \begin{equation}
     -\log r \delta_{ij} + \frac{(x_i-y_i)(x_j-y_j)}{r^2}  =- \lp \delta_{ij} + x_j \partial_{x_i} \rp \log r + y_j \partial_{x_i} \log r.
 \end{equation}
 for $1 \le i,j, \le 2.$
\end{lemma}
This lemma readily implies that the kernel $K_T$ associated with $\mathcal{K}_T$ is of the general form
\begin{align}
    K_T(\br,\br') &= \sum_{0\le j_1+j_2+j_3 \le 3}\, \sum_{0 \le i_1+i_2+i_3 \le 1} A_{j_1,j_2,j_3}^{i_1,i_2,i_3}(\br)r_{1}^{i_1}r_{2}^{i_2}r_3^{i_3}\partial_{r_{1}}^{j_1}\partial_{r_{2}}^{j_2}\partial_{r_{3}}^{j_3} \log\|\br-\br'\| B_1(\br')\label{eqn:KT_gen_form}\\
    &+\sum_{0\le j_1+j_2+j_3 \le 3}\, \sum_{0 \le i_1+i_2+i_3 \le 1} A_{j_1,j_2,j_3}^{i_1,i_2,i_3}(\br)r_{1}^{i_1}r_{2}^{i_2}r_3^{i_3}\partial_{r_{1}}^{j_1}\partial_{r_{2}}^{j_2}\partial_{r_{3}}^{j_3} \log\|\br-\br'\| B_2(\br') \nonumber \\
    &+\sum_{0\le j_1+j_2+j_3 \le 3}\, \sum_{0 \le i_1+i_2+i_3 \le 1} A_{j_1,j_2,j_3}^{i_1,i_2,i_3}(\br)r_{1}^{i_1}r_{2}^{i_2}r_3^{i_3}\partial_{r_{1}}^{j_1}\partial_{r_{2}}^{j_2}\partial_{r_{3}}^{j_3} \log\|\br-\br'\| B_3(\br'),\nonumber
\end{align}
where the $4\times 4$ coefficient tensors $A_{j_1,j_2,j_3}^{i_1,i_2,i_3}$ depend on $\bn(\br),S(\br)$ and $H(\br),$ and the entries of the $4 \times 4$ matrix-valued functions $B^j(\br')$ are polynomials in $\br'$ and $\bn(\br').$

In light of (\ref{eqn:KT_gen_form}), if~$\rho_j$ is the subset of entries of the density corresponding to box~$B_j$, the vector~$A_{i,j}\rho_j$, can be written as a combination of functions of the form
\begin{align}
\phi(\br):=\sum_{0\le j_1+j_2+j_3 \le 3}\, \sum_{0 \le i_1+i_2+i_3 \le 1} C_{j_1,j_2,j_3}^{i_1,i_2,i_3}r_{1}^{i_1}r_{2}^{i_2}r_3^{i_3}\partial_{r_{1}}^{j_1}\partial_{r_{2}}^{j_2}\partial_{r_{3}}^{j_3} \sum_{k=1}^N\log\|\br-\br'_k\| \tilde \rho_k ,
\end{align}
evaluated at the targets in box~$i,$ where~$\tilde \rho_k$ includes factors like~$\omega_k$ and~$B(\bs r_k')$, with some terms multiplied by diagonal operators corresponding to surface properties. Each~$\phi$ can be expressed as a differential operator~$\cC$ applied to the function
\begin{align}\label{eqn:psi_def}
    \psi(\br):= \sum_{k=1}^N\log\|\br-\br'_k\|\tilde \rho_k.
\end{align}
The following lemma enables the explicit expansion of (\ref{eqn:psi_def}) in terms of spherical harmonics. For ease of exposition, we assume that all spherical coordinate systems in this section are centered at~$\bx_j$.
\begin{lemma}\label{lem:log_exp}
If~$\bx = (R,\theta,\phi)$ in spherical coordinates and~$R>r=\Vert \by \Vert$, then
\begin{equation}
    \log \Vert \bx - \by\Vert = \log R - \sum_{n=1}^\infty \frac{a_n}{R^n}\sum_{l=0}^n L_{nl} \sqrt{\frac{4\pi}{2\ell+1}}\sum_{m=-l}^l Q^l_m Y^l_m(\theta,\phi), \label{eq:single_source}
\end{equation}
where~$a_n=\frac{r^n}{n}$,~$L_{nl}$ is defined below and~$Q^l_\cdot=Q^l_\cdot(\by)$ is a normalized vector of length~$2l+1$ for each~$l\geq0$.
\end{lemma}
\begin{proof}
The two-dimensional multipole expansion (see~\cite{joslin1983multipole}) gives that
\begin{equation}
    \log \Vert \bx - \by\Vert = \log R - \sum_{n=1}^\infty \frac{\cos\lp n\theta_s\rp r^n}{R^n n} \label{eq:thetas_exp},
\end{equation}
where~$\theta_s$ is the angle between~$\bx$ and~$\by$.
We now seek to express the right-hand side of the above equation as a function of the polar angles~$\theta$ and~$\phi$. First we write~$\cos(n\theta_s)$ as a sum of Legendre polynomials:
\begin{equation}
    \cos(n\theta_s) = T_n(\arccos(\theta_s)) = \sum_{l=0}^n L_{nl} P_l(\arccos(\theta_s))  \label{eq:cheb_to_leg},
\end{equation}
where~$T_n$ is a Chebyshev polynomial,~$P_l$ is a Legendre polynomial, and the coefficients~$L_{nl}$ are
\begin{equation}
    L_{nl} = \begin{cases} 1 & n=l=0\\
    \frac{\sqrt{\pi}}{2\Lambda(n)} & n=l>0\\
    \frac{-n(l+1/2)}{(n+l+1)(n-l)}\Lambda\lp \frac{n-l-2}2\rp\Lambda \lp \frac{n+l-1}2\rp & n>l\text{ and } n+l \text{ is even}\\
    0 &\text{otherwise}
    \end{cases},
\end{equation}
with~$\Lambda(z) = \frac{\Gamma(z+1/2)}{\Gamma(z+1)}$ (see~\cite{alpert1991fast}). Noting that~$P_l(\arccos(\theta_s))=\sqrt{\frac{4\pi}{2l+1}}Y_l^0(\theta_s,0)$ and using the rotation formula for spherical harmonics, \eqref{eq:cheb_to_leg} becomes
\begin{equation}
     \cos(n\theta_s) = \sum_{l=0}^n L_{nl} \sqrt{\frac{4\pi}{2l+1}}Y_l^0(\theta_s,0)=\sum_{l=0}^n L_{nl} \sqrt{\frac{4\pi}{2l+1}} \sum_{m=-l}^l Q^l_m Y^m_l(\theta,\phi).
\end{equation}
With this choice of normalization, the spherical harmonics are an orthonormal basis for~$L^2$ on the sphere, and so~$Q^l_\cdot$ is a vector of norm~$1$.
Substituting this expression into~\eqref{eq:thetas_exp} gives~\eqref{eq:single_source}.
\end{proof}

In order to truncate the sums, we must control the behavior of~$L_{nl}$ for large~$n$.
Lemma 2.4 in~\cite{alpert1991fast} states that 
\begin{equation}\label{eq:lambdabd}
  \frac{e^{-1/2}}{\sqrt{z+1}}\leq |  \Lambda(z)| \leq \frac{e}{\sqrt{|z+1|}}
\end{equation}
for all~$z\geq0$. Using these inequalities, the following lemma establishes a bound for $L_{nl}$.
\begin{lemma}\label{lem:Lnl_bd}
    There is a~$C>0$ such that~$|L_{nl}|\leq C\sqrt{n+1}$ for all~$0\leq l\leq n$.
\end{lemma}
\begin{proof}
If~$l\leq n-2$, we can use \eqref{eq:lambdabd} to bound
\begin{equation}
    |L_{nl}|\leq \frac{n(l+1/2)}{(n+l+1)(n-l)}\frac{2e^2}{\sqrt{(n-l-2+2)(n+l-1+2)}}\leq  \frac{n(n-3/2)e^2}{(n+1)^{3/2}\sqrt{2}}\leq \frac{e^2}{\sqrt{2}}\sqrt{n+1}\,.
\end{equation}
We can also bound the remaining non-zero terms using \eqref{eq:lambdabd}:
$
    L_{nn}\leq \frac{\sqrt{\pi e}}2 \sqrt{n+1}.
$
Combining these bounds gives the result.
\end{proof}

The previous two lemmas immediately imply the following decomposition of $\psi.$

\begin{corollary}\label{cor:psi_lowrank}
        Let $\psi$ be as defined in (\ref{eqn:psi_def}). Then
    \begin{align}\label{eqn:psi_expans}
        \psi(\br) ={\sqrt{4\pi}}\log(\|\br\|)\alpha_{0,0,0}-\sum_{\ell=0}^\infty \sum_{m=-\ell}^\ell Y_m^\ell(\theta_{\br},\phi_{\br}) \sum_{n=\max(\ell,1)}^\infty\frac{\rho^n}{n\|\br\|^n}\alpha_{n,\ell,m}
    \end{align} 
    with
    $$\alpha_{n,\ell,m} =\sqrt{\frac{4\pi}{2\ell+1}}L_{n,\ell} \sum_{k=1}^N\frac{\|\br_k'\|^n}{\rho^n}Q_m^\ell(\theta_{\br_k'},\phi_{\br_k'}) \tilde \rho_k.$$
    Moreover, there exists a global constant $C$ independent of $n,\ell$ and $m$ such that the coefficients $\alpha_{n,\ell,m}$ are bounded by
    $$|\alpha_{n,\ell,m}| \le C \sqrt{\frac{n+1}{\ell+1}} \,\sum_{k=1}^N |\tilde \rho_k|.$$
    In particular, if the series is truncated at $\ell=L,$ the absolute value of the truncation error, $\mathcal{E}_L,$ is bounded by
    $$\mathcal{E}_L \le C'' L^2 \log(\|\br \|/\rho)^{-1}\left(\frac{\rho}{\|\br\|}\right)^L\sum_{k=1}^N |\tilde \rho_k|$$
    for some absolute constant $C''.$ If the series is truncated at~$\ell,n\leq L$, then the absolute value of the truncation error satisfies the same estimate.
    \end{corollary}
    \begin{proof}
        The expansion~\eqref{eqn:psi_expans} and the bounds on~$\alpha_{n,\ell,n}$ are a direct consequence of the previous lemmas. The truncation bound can be obtained by bounding the remainder in terms of an integral and using standard properties of the incomplete~$\Gamma$ function.
    \end{proof}

When we build~$X_j$, it will generate an interpolant of~$\psi$. In what follows we use the following approximation theory result to bound the interpolation error.
\begin{theorem}\label{thm:interp_bd}
    Consider a function $g:[0,1] \to \mathbb{R}$ of the form
    $$g(x) = \sum_{n=1}^\infty x^n g_n,$$
    with $|g_n| \le C n^a$ for some constants $a> -1$ and $C$. Let $\hat{x}_1,\cdots,\hat{x}_P$ be Chebyshev nodes of the first kind \cite{trefethen2019approximation} shifted and scaled to live in~$[0,1/(1+c)]$ for~$c>0$. If, $|g(\hat{x}_i)|<\epsilon$ for all then 
    \begin{equation}
        |g(x)|\leq D (\log(P+1)+1)\lp \epsilon + \frac{\zeta^{-P}}{\zeta-1}\rp, 
    \end{equation}
    where~$\zeta = c+1 + \sqrt{c^2+2c}$ and $D$ is a constant that depends only on $a,c,$ and $C$.
\end{theorem}
\begin{proof}
To bound $g$, we seek a polynomial $g_P$ of degree less than $P$, which approximates~$g$. To bound the approximation error, we first bound $g$ in a disk centered at 0. If~$1/(1+c) <r<1,$ then we can bound $|g(z)|$ as follows,
\begin{equation*}
    \max_{z\in B_r(0)}|g(z)| \le \tilde{C} \int_0^\infty r^n n^a\,{\rm d}n \le \tilde{C} \frac{\Gamma(a+1)}{\log(1/r)^{a+1}}=:M_r.
\end{equation*}
If we let $h(z) := g((z+1)/2(1+c))$ and $\lambda = 2r(1+c)-1$, then this result implies that $h$ is bounded by~$M_r$ on the Bernstein ellipse with $\rho+1/\rho = 2 \lambda.$ It follows that $\rho = \lambda + \sqrt{\lambda^2-1},$ and that there exists (see Chapter 8 of \cite{trefethen2019approximation}) a polynomial $h_d$ of degree less than $d$ such that
\begin{equation*}
    \| h_P - h\|_{L^\infty([-1,1])} \le 2M_r \frac{\rho^{-P}}{\rho-1}.
\end{equation*}
If $g_P(z) := h_P( 2z(1+c)-1),$ then this implies that
\begin{equation*}
    \sup_{z\in[0,1/(1+c)]} |g_P(z)-g(z)| \le \tilde C\frac{\Gamma(a+1)}{\log(1/r)^{a+1}}\frac{\rho^{-P}}{\rho-1}.
\end{equation*}
If we let~$r=(1/(1+c)+1)/2=(2+c)/2(1+c)$, then we can write~$\rho$ in the more explicit form 
\begin{equation*}
   \rho = c+1 + \sqrt{c^2+2c}.
\end{equation*}

Finally, the assumption that~$|g(\hat x_i)|<\epsilon$ implies that $|g_d(\hat x_i)|\leq \epsilon + C\frac{\Gamma(p+1)}{\log(1/r)^{p+1}}\frac{\rho^{-d}}{\rho-1}$. Since $g_d$, is a polynomial, we can bound it on the whole interval~$[0,1/(1+c)]$ by these values time the Lebesgue constant $\Lambda_d$. Since the Lebesgue is bounded by $\frac2\pi \log (d+1)+1$, this proves the result.

\end{proof}

\begin{theorem}\label{thm:vinterp_err}
    Let~$\{\bs p'_{p}\}$ be a collection of quadrature nodes on the unit sphere that exactly integrate $Y_{m}^\ell(\bs \theta) \overline{Y_{m'}^{\ell'}}(\bs \theta)$ for $0\le \ell,\ell' \le L,$ and $-\ell\le m \le \ell,$ $-\ell'\le m' \le \ell'.$ Let~$\hat x_1,\ldots \hat x_P$ be~$P$ Chebyshev nodes of the first kind in the interval~$[0,1/{(1+c)\rho}]$. Also let
    \begin{equation}\label{eq:vintep_exp}
            v(\br) =\sum_{\ell=0}^\infty \sum_{m=-\ell}^\ell Y_m^\ell(\theta_{\br},\phi_{\br}) \sum_{n=\max(\ell,1)}^\infty\frac{\rho^n}{\|\br\|^n}\gamma_{n,\ell,m}\,,
        \end{equation}
        where~$|\gamma_{n,\ell,m}|< C/\sqrt{n+1}$.
        
    If~$\bs p_{p,p'} = \frac{\bs p'_{p}}{\hat x_p}$ and~$|v(\bs p_{p,p'})|<\epsilon$ for all~$p$ and~$p'$, then
    \begin{equation}\label{eq:v_bound}
        |v(\br)| \leq  D L^2(\log(P+1)+1)\lp \epsilon + \frac{\zeta^{-P}}{\zeta-1}\rp + D L^2\log(1+c)^{-1} (1+c)^{-L}
    \end{equation} for all~$\br\in B_R(\bx_j)\setminus B_{(1+c)\rho}(\bx_j)$, where~$D$ is a constant independent of~$v$ and $L$, and $\zeta<1$ is the constant given in Theorem \ref{thm:interp_bd}.
\end{theorem}
\begin{proof}
To bound~$v$, we split it into the~$\ell\leq L$ terms, denoted~$v_1$, and the remainder~$v_2$. 
 For all~$0\leq |m|\leq \ell \leq L$, we let
\begin{equation}
    v_1^{m,\ell}(r) = \int_S v_1^{m,\ell}(r,\bs \theta) \overline{Y_{m}^{\ell}}(\bs \theta) dA(\bs \theta) = \sum_p v_1^{m,\ell}(r,\bs \theta_p) \overline{Y_{m}^{\ell}}(\bs \theta_p) = \sum_{n=\ell}^\infty \frac{\rho^n}{n r^n} \gamma_{n,\ell,m}\,.
\end{equation}
be the spherical Fourier coefficient of~$v_1$ at each~$r$. The choice of nodes~$\bs p_{p,p'}$ and assumption that~$v_1$ is small at these nodes guarantees that
\begin{equation}
      |v_1^{m,\ell}(1/\hat x_{p'})|\leq \epsilon \sum_p |\tau_p|,
\end{equation}
where~$\tau_p$ are the quadrature weights associated with the spherical quadrature nodes~$\bs p'_p$. Applying Theorem~\ref{thm:interp_bd} with~$x=\rho/r$ gives that
\begin{equation}
    | v_1^{m,\ell}(r)|\leq  C (\log(P+1)+1)\lp \epsilon\sum_p |\tau_p| + \frac{\zeta^{-P}}{\zeta-1}\rp
\end{equation}
for all~$r\in [(1+c)\rho,R]$. Since
\begin{equation}
    v_1(r,\bs \theta) = \sum_{\ell=0}^N \sum_{m=-\ell}^\ell v_1^{m,\ell}(r)Y_\ell^m(\bs \theta),
\end{equation}
this result implies that~$|v_1(r,\theta)|\leq DL^2 (\log(P+1)+1)\lp \epsilon\sum_p |\tau_p| + \frac{\zeta^{-P}}{\zeta-1}\rp$  for all~$r\in [(1+c)\rho,R]$.

The remainder~$v_2$ can easily be bounded by Corollary~\ref{cor:psi_lowrank} and so this proves the result.
\end{proof}
The result can be easily extended to bound~$\cC[v]$ using Lemmas~\ref{lem:y_mult} and \ref{lem:ylm_der} in Appendix~\ref{app:Ylm_id}.
\begin{corollary}
    If~$v$ satisfies the same assumptions as Theorem~\ref{thm:vinterp_err}, then~$\cC[v]$ satisfies a bound equivalent to~\eqref{eq:v_bound} multiplied by~$R$ and~$L$ raised to powers that depend on the order of the polynomials and differential operators in~$\cC$.
\end{corollary}

To construct the matrix~$X_j$, we choose~$R$ large enough so that~$B_R(\bx_j)$ covers~$\Gamma,$ and lay down proxy points~$\bs p_1,\ldots \bs p_P$ covering~$B_{R}(\bx_j) \setminus B_{(1+c)\rho}(\bx_j)$. We then construct the matrix~$D$ with entries
\begin{equation}\label{eq:Dcompress}
    D_{p,k} = \log\|\bs p_p-\br'_k\| - \log\|\bs p_p-\bx_j\|
\end{equation}
for all~$k$ such that~$\br'_k\in B_j$.  We also append a row of ones to~$D$.
Corollary~\ref{cor:psi_lowrank} implies that~$D_{p,k}$ admits an approximate low rank factorization with a rank that scales as~$\lp\log \epsilon / \log(1+c)  \rp^3$. In particular, for any~$\epsilon>0$, we can construct its interpolative decomposition~\cite{cheng2005compression}, which finds a matrix~$\tilde D$ composed of columns~$k_1,\ldots, k_{N_s}$ of~$D$ and a matrix~$\tilde X_j$ such that
\begin{equation}\label{eq:interp_decomp}
    \|D - \tilde D \tilde X_j\|_{\infty} \leq \epsilon.
\end{equation}
Further, the factorization can be chosen so that columns $k_1,\ldots, k_{N_s}$ of~$\tilde X_j$ form the identity matrix. To see the advantage of this property, we let~$\hat \rho_s = \sum_{k}(X_j)_{s,k} \tilde \rho_k $ and define the function
\begin{equation}
        \tilde \psi(\br):= \sum_{s=1}^{N_s}\log\|\br-\br'_{k_s}\| \hat \rho_s = \sum_{s=1}^{N_s}\lp\log\|\br-\br'_{k_s}\|-\log\|\br-\bx_j\|\rp \hat \rho_s + \log\|\br-\bx_j\|\sum_{s=1}^{N_s} \hat \rho_s .
\end{equation}
By construction~$\sum_s \hat \rho_s = \sum_k \tilde\rho_k$, and so~$\tilde\psi(\bs p_p) = (\tilde D \tilde X_j \rho)_k + \log\|\bs p_p-\bx_j\| \sum_k \tilde\rho_k$. The definition~\eqref{eq:interp_decomp} thus gives that~$|\tilde \psi(\bs p_p) - \psi(\bs p_p)|<\epsilon \max|\tilde p_k|$ for all~$p$. We now argue that, provided proxy nodes are well-chosen, this is sufficient to guarantee that~$\phi \approx \cC [\tilde \psi]$ and so~$\tilde X_j$ can be used to construct the desired~$X_j$. 

The function~$v=\psi - \tilde\psi$ satisfies the assumptions of Theorem~\ref{thm:vinterp_err}. Since~$K_T$ is a combination of functions of the form of~$\cC[\psi]$, for various~$\cC$ (multiplied by properties such as the mean curvature), Theorem~\ref{thm:vinterp_err} directly implies the following corollary, which is the main result of this section.
\begin{corollary}
    Let~$\tilde X_j$ be the low rank factorization defined in~\eqref{eq:interp_decomp}. If~$ X_j = \tilde X_j \otimes I_4$, then~$A_{ij} \approx \tilde A_{ij}  X_j$ for all~$i$ such that~$B_i$ is well-separated from~$B_j$. Further, the approximation error is~$O\lp L^5\log L  (\epsilon+\zeta^{-L})\rp$.
\end{corollary}
The $5$th power in the error bound is due to the formulas in Lemma~\ref{lem:ylm_der} and the fact that the differential operators~ $\cC$ are at most third order.

Our method for computing~$X_j$ can be summarized in the following algorithm
\begin{algorithm}[H]
\caption{Computation of the matrix $X_j$}
\label{alg:1}
\mbox{}
\begin{enumerate}
\item Lay down proxy points~$\bs p_{p}$ lying in a sequence of spheres centered at $\bx_j$ with radii covering~$[(1+c)\rho,\infty),$ as described in Theorem~\ref{thm:vinterp_err}.
\item Evaluate the matrix~$D$ with entries $D_{p,k} = \log\|\bs p_{p}-\br'_k\| - \log\|\bs p_{p}-\bx_j\|$ for every~$\br'_k$ in~$B_j$ and every proxy point~$\bs p_{p}$.
\item Append a row of ones to~$D$.
\item Evaluate~$K_T(\br_k,\br_{k'})$ for every pair of points with~$\br_k \in B_{(1+c)\rho}(\bx_j)\setminus B_j$ and~$\br_{k'}\in B_j$ and append these rows to~$D$.
\item Compute the interpolative decomposition~$D=\tilde D\tilde  X_j$ where columns~$k_1,\ldots, k_{N_s}$ of~$\tilde X_j$ form an identity block.
\item Define $ X_j =\tilde X_j \otimes I_4$.
\end{enumerate}
\end{algorithm}

Step 4 in this algorithm ensures that the matrix~$X_j$ can also be used for compressing the interaction of~$B_j$ with~$B_i$, even when they are not well-separated. In some compression schemes, such as~\cite{minden2017recursive,sushnikova2023fmm}, this step should be omitted.

The construction of the $E_i$ matrices is analogous to the procedure for constructing $X_j$'s, except that this time the differential operator~$\cC$ involves derivatives with respect to the target positions $\bs r_k$ in the ball~$B_\rho(\bx_i)$. We must therefore build the interpolative decomposition of a matrix containing all of our kernels, rather than just the archetypal logarithm. This procedure is summarized in the following algorithm.
\begin{algorithm}[H]
\caption{Computation of the matrix $E_i$}
\label{alg:2}
\mbox{}
\begin{enumerate}
\item Lay down proxy points~$\bs p_{p}$ lying in a sequence of spheres centered at $\bx_i$ with radii covering~$[(1+c)\rho,\infty),$ as described in Theorem~\ref{thm:vinterp_err}.
\item Evaluate the block matrix~$D$ with entries $D_{k,p} = K_T(\br_k,\bs p_{p} )$ for every~$\br_k$ in $B_i$ and every proxy point $\bs p_{p}$, where the different components of~$K_T$ are organized into different rows.
\item Evaluate~$K_T(\br_k,\br_{k'})$ for every pair of points with~$\br_k \in B_{(1+c)\rho}(\bx_i)\setminus B_i$ and~$\br_{k'}\in B_j$ and append these rows to~$D$. 
\item Compute the interpolative decomposition~$D=\tilde E_i \hat D$ where rows~$k_1,\ldots, k_{N_s}$ of~$\tilde E_i$ form an identity block.
\item Define $ E_i =\tilde E_i \otimes I_4$.
\end{enumerate}
\end{algorithm}
As for~$X_j$, the interaction with points in neighboring boxes should be omitted in some compression schemes.

\section{Numerical Illustrations}\label{sec:numill}
In this section we illustrate the effectiveness of our compression scheme and test our solver on a few different surfaces.

\subsection{Compression}\label{sec:compress_test}
We begin by testing the compression scheme described in Section~\ref{sec:build_compress} with~$\rho=1$, $R=1000$, and~$c=1$. We distribute 5 000 points uniformly in~$B_1(0)$ and 5 000 points uniformly in~$B_{1000}(0) \setminus B_{2}(0)$. For the purposes of this test, we set the surface parameters that appear in~$K_T$ to be random numbers uniformly distributed in~$[-1,1]$.

We lay down proxy shells in the method summarized in Theorem \ref{thm:vinterp_err}. On each shell, we use $L$th order Lebedev nodes. We build the matrix $D$ given by \eqref{eq:Dcompress} and build an interpolative decomposition with tolerance $10^{-10}$. The resulting maximum errors in the approximations of the logarithm of the distance between points is shown in \figref{fig:shell_compress_log}. The resulting maximum errors in the approximation of $K_T$ are shown in \figref{fig:shell_compress_stoke}.

\begin{figure}
    \centering
\begin{minipage}{.45\textwidth}
  \centering  \includegraphics[width=0.95\linewidth]{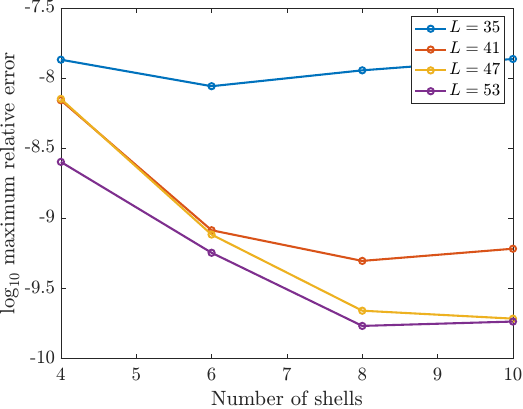}
  \captionof{figure}{The maximum error in the compressed approximation of the logarithmic kernel found in the test described in Section~\ref{sec:compress_test}. It is plotted as a function of the number of proxy shells and the order of the quadrature scheme on each shell.}
  \label{fig:shell_compress_log}
\end{minipage}\quad%
  \begin{minipage}{.45\textwidth}
  \centering
  \includegraphics[width=0.95\linewidth]{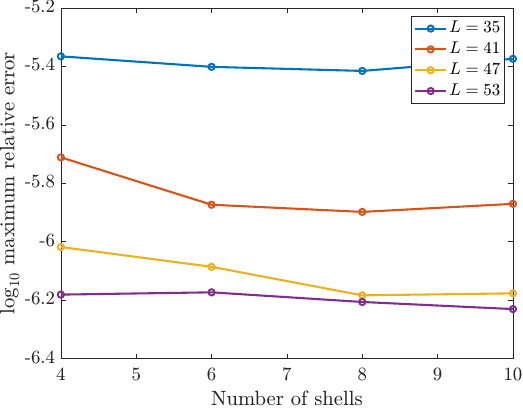}
  \captionof{figure}{The maximum error in the compressed approximation of the remainder kernel found in the test described in Section~\ref{sec:compress_test}. It is plotted as a function of the number of proxy shells and the order of the quadrature scheme on each shell.}
  \label{fig:shell_compress_stoke}
\end{minipage}
\end{figure}

\subsection{Convergence test}
Having demonstrated the effectiveness of our compression scheme, we now test the convergence of our solver. In this example, we choose the surface~$\Gamma$ to be the surface parameterized by
\begin{equation}
    \br(u,v) = \begin{pmatrix}
        1 & 0 & 1/2\\ 0&1&0\\0&0&1
    \end{pmatrix} \left[ \begin{pmatrix}
        \lp 2+\frac34\cos(v)+\frac34\cos(3u)\rp\cos(u) \\
        \lp 2+\frac34\cos(v)+\frac34\cos(3u)\rp\sin(u) \\
        \frac34\sin(v)
    \end{pmatrix} - \begin{pmatrix}
        1.5\\0\\0
    \end{pmatrix}\right]+\begin{pmatrix}
        1.5\\0\\0
    \end{pmatrix}.
\end{equation}
To test the solver, we choose the right hand side so that the true solution is $u= P(z,x,y)^T$ and $p=z$. This true solution is shown in \figref{fig:exact}. We construct the right hand side using the 24th order spectral differentiation implements in the surfacefun package \cite{fortunato2024high}. We then discretize the surface using an increasing number of 8th order tensor product Gauss-Legendre quadrilateral patches and construct a fast direct solver for \eqref{eqn:bie_red} and fast apply for evaluating the representation and recovering the velocity $\bs u$ with compression and quadrature tolerances set to $10^{-7}$. The resulting relative Frobenius errors in the velocity (\figref{fig:slant_err}) show that the error in our solver is comparable to the requested tolerance. In \figref{fig:slant_time} we show the time to construct the fast direct solver for \eqref{eqn:bie_red}, which is the dominant cost in our solver. We can see that the cost scales linearly as we refine our surface discretization. In all of these experiments, the number of degrees of freedom is three times the number of discretization nodes, $n_{\rm{pts}}$.

As another example, we solve the surface Stokes equations on an ellipsoid with axis lengths 1.5, 1, and 1. To remove the null space discussed in remark \ref{rem:killing}, we add the term $\bs u$ to the first equation
\begin{equation*}
    -\frac12 \cP \divg \lp\gradg \bu + (\gradg \bu)^T\rp + \bs u+ \gradg p = {\bs f}
\end{equation*}
and modify \eqref{eqn:bie_red} accordingly.
For this test, we discretize the ellipsoid with 288 eighth-order triangular patches (12 960 points), set our integration tolerance to $10^{-9}$ and compression tolerance to $10^{-8}$. The relative Frobenius error in the manufactured solution test (\figref{fig:ell_exact}) with $\bs u= P(z,x,y)^T$ and $p=z$ was $ 3.2982\times 10^{-8}.$

\begin{figure}
    \centering
\begin{minipage}{.45\textwidth}
    \centering
    \includegraphics[width=0.95\linewidth]{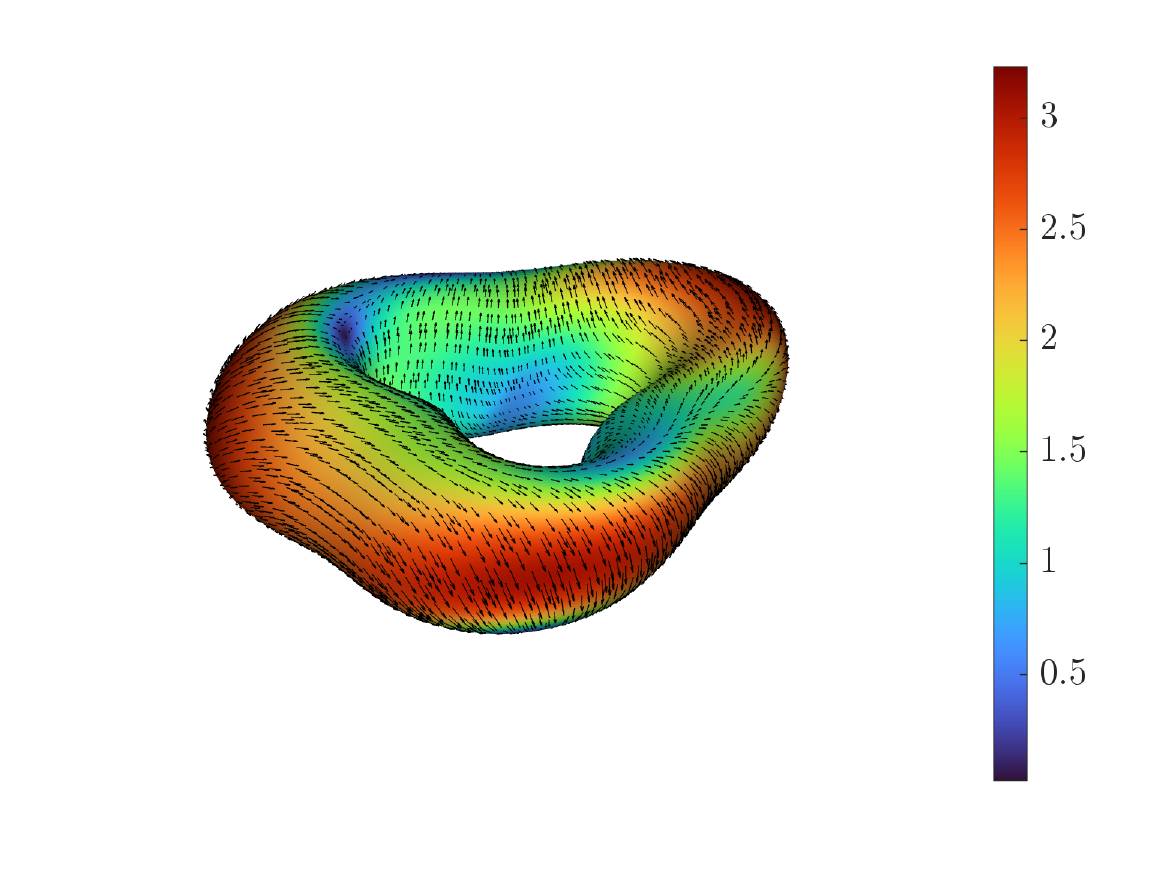}
    \captionof{figure}{The reference solution~$\bs u$ used for the convergence test with errors in \figref{fig:slant_err}. The color indicates the norm of~$\bs u$.}
    \label{fig:exact}
    \end{minipage}\quad%
  \begin{minipage}{.45\textwidth}
  \centering
\includegraphics[width=0.95\linewidth]{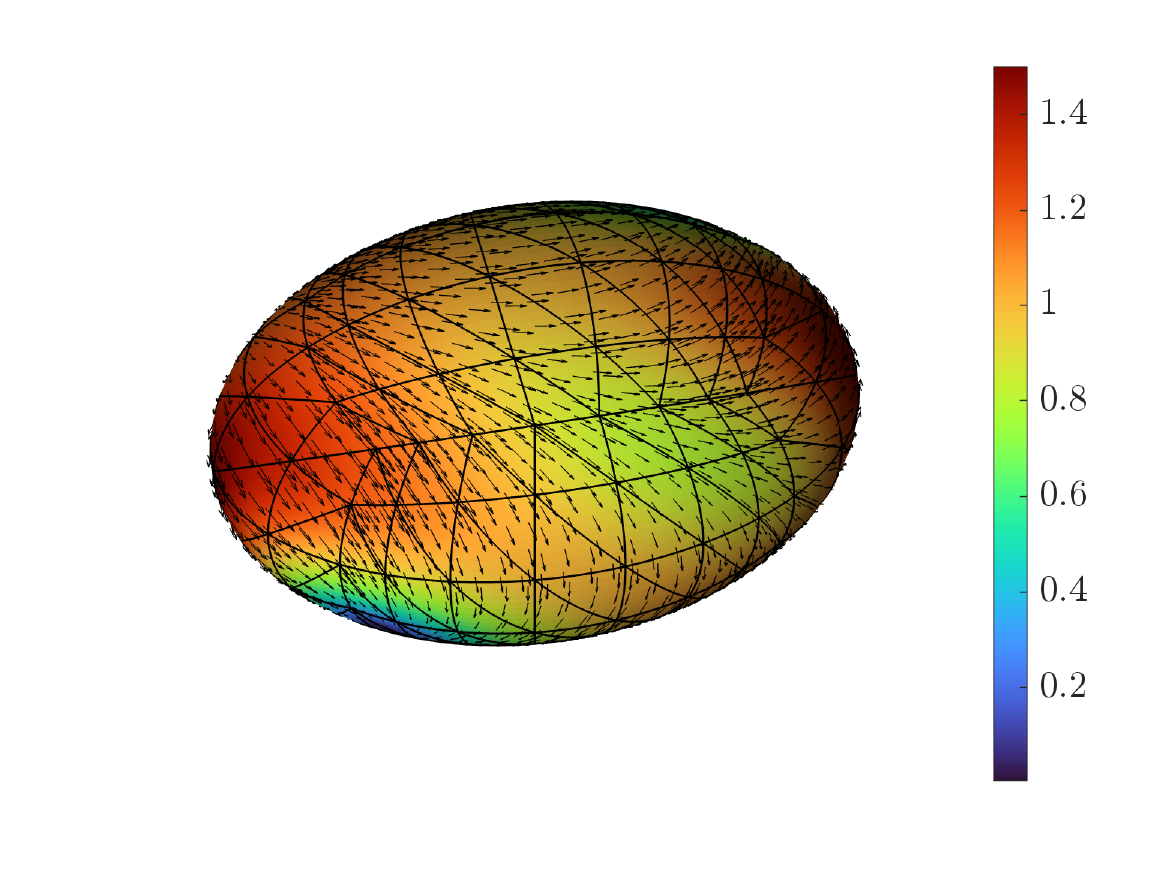}
  \captionof{figure}{The reference solution~$\bs u$ used for the ellipsoid test. The black lines show the mesh and the color indicates the norm of~$\bs u$.}
  \label{fig:ell_exact}
  \end{minipage}
\end{figure}

\begin{figure}
    \centering
\begin{minipage}{.45\textwidth}
  \centering  \includegraphics[width=0.95\linewidth]{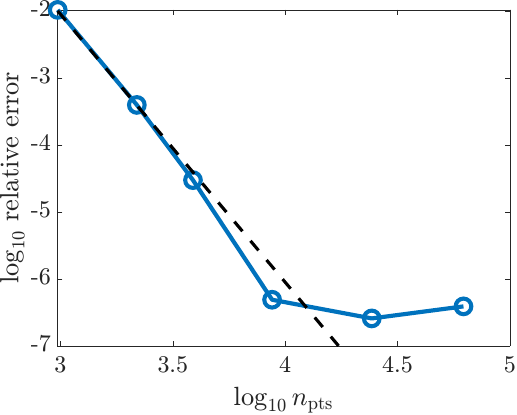}
  \captionof{figure}{The relative error~$L^2$ error in~$\bu$ for an analytical solution test on a slanted torus. The black dashed line indicates 8th order convergence. The true solution is shown in \figref{fig:exact}.}
  \label{fig:slant_err}
\end{minipage}\quad%
  \begin{minipage}{.45\textwidth}
  \centering
  \includegraphics[width=0.95\linewidth]{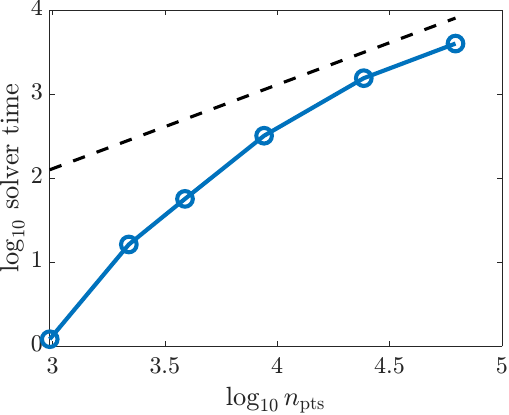}
  \captionof{figure}{The times to build the inverse for \eqref{eqn:bie_red} versus the number of surface discretization nodes. The tests were run on an Apple MacBook Pro with an M2 Max chip. The black dashed line indicates $O(n_{\rm{pts}})$ time.}
  \label{fig:slant_time}
\end{minipage}
\end{figure}

As a final example, we consider the surface Stokes equation on the surface is parameterized by
\begin{equation}
    \br(u,v) = \begin{pmatrix}
        1 & 0 & 1/2\\ 0&1&0\\0&0&1
    \end{pmatrix} \left[ \begin{pmatrix}
        \lp \frac32+\lp \frac12+\frac1{10}\cos(5v)\rp \cos(v)\rp\cos(u) \\
        \lp \frac32+\lp \frac12+\frac1{10}\cos(5v)\rp \cos(v)\rp\sin(u) \\
        \lp \frac12+\frac1{10}\cos(5v)\rp\sin(v)
    \end{pmatrix} - \begin{pmatrix}
        1.5\\0\\0
    \end{pmatrix}\right]+\begin{pmatrix}
        1.5\\0\\0
    \end{pmatrix}.
\end{equation}
We discretize this surface using 16 384 points arranged in 256 eighth-order quadrilateral patches (\figref{fig:star_mesh}).
We choose the right hand side $\bs f =0$ and choose $g$ to be a difference of functions of the form $e^{-3(\bx-\bx_0)^2}$ centered at $\bx_0=(-1.6,1,0.2)^T$ and $\bx_0=(1.5,0,0.4)^T$. We make the coefficient of the second term negative and scale the two Gaussians so that $g$ satisfies the mean zero consistency condition. The resulting solution is shown in \figref{fig:fig_ex_sol}.

\begin{figure}
    \centering
\begin{minipage}{.45\textwidth}
  \centering  \includegraphics[width=0.95\linewidth]{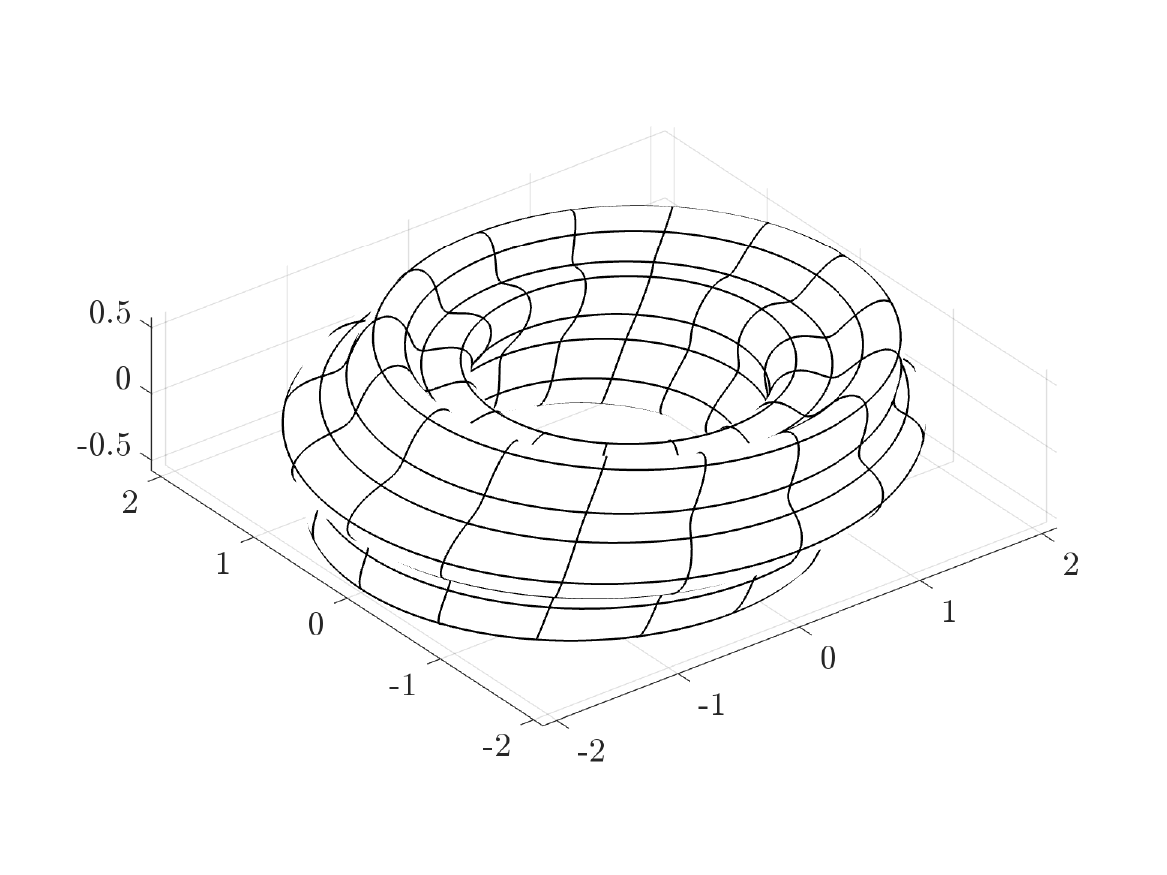}
  \captionof{figure}{The mesh used to discretize the surface in \figref{fig:fig_ex_sol}.}
  \label{fig:star_mesh}
\end{minipage}\quad%
  \begin{minipage}{.45\textwidth}
  \centering
  \includegraphics[width=0.95\linewidth]{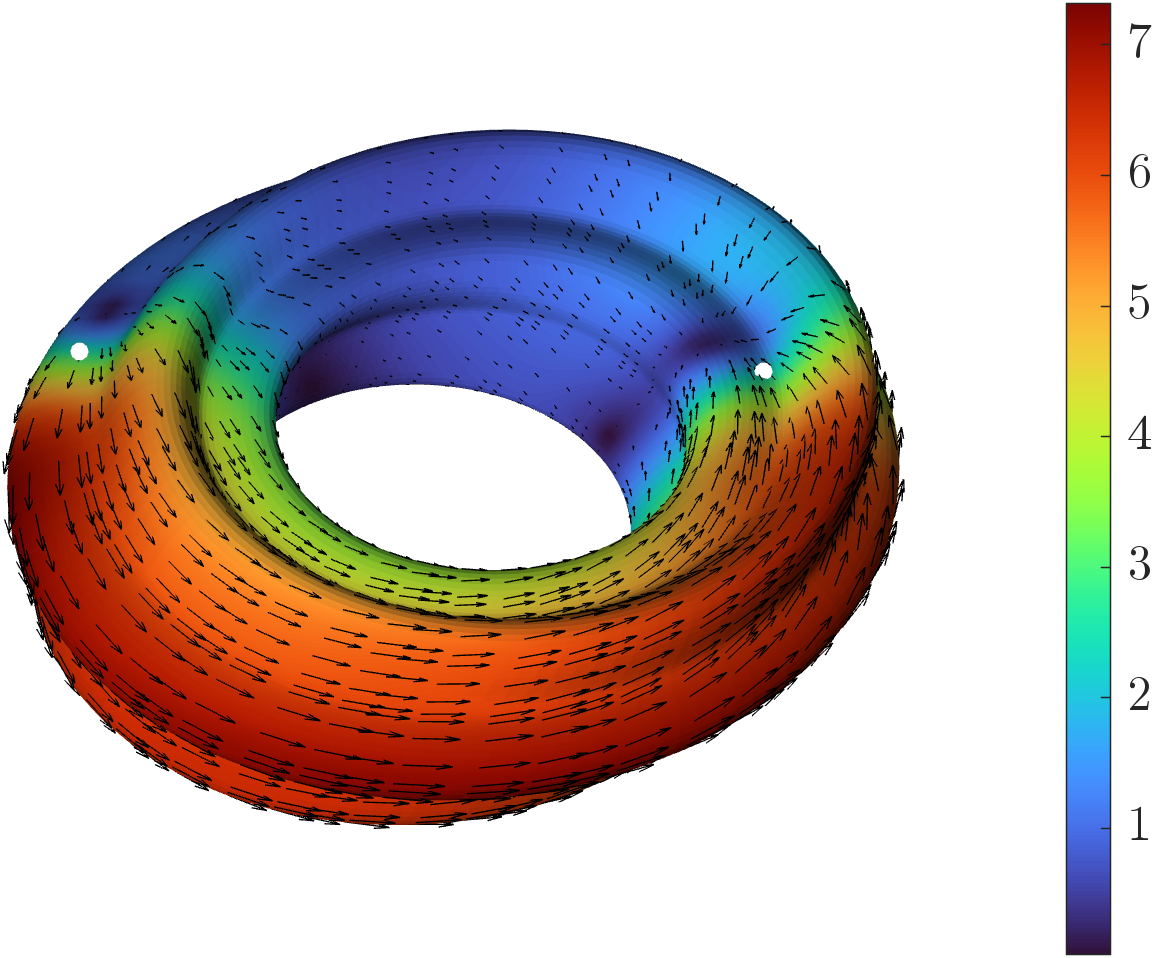}
  \captionof{figure}{An example solution generated by a fluid source and sink centered at the white dots.}
    \label{fig:fig_ex_sol}
\end{minipage}
\end{figure}

\section{Conclusion}

In this paper we have described a numerical approach for solving the surface Stokes equations on arbitrary smooth manifolds embedded in three dimensions. The method is based on introducing suitable parametrices that convert the surface PDEs associated with the Stokes equations into a coupled system of integral equations for an unknown vector-valued surface density. We have shown that this resulting system is a Fredholm integral equation of the second kind, which upon discretization typically leads to linear systems which are well-conditioned under mesh refinement. However, the resulting discrete systems are dense, with non-zero entries between degrees of freedom corresponding to even well-separated points. To overcome this difficulty, we adapted a standard fast algorithm for forming compressed representations of the inverse of these linear systems. This broad class of methods is based on the observation that the rank of interaction between  physically well-separated sets of points is low. For our particular kernels, we prove this fact and provide the bounds for the interaction ranks. These estimates are potentially interesting in their own right, and the proof techniques apply to a much broader class of problems. In fact the fast direct solver demonstrated here can be used to efficiently solve time dependent surface Stokes problems, nonlinear parabolic equations (see \cite{fortunato2024high}). The presented method can also be extended to solve other surface PDEs, such as the matrix-valued time-harmonic elastodynamic equation or fourth order thin-plate problems.

\section{Acknowledgements}
J.G.H. was supported in part by a Sloan Research Fellowship. This research was supported in part by grants from the NSF (DMS-2235451) and Simons Foundation (MPS-NITMB-00005320) to the NSF-Simons National Institute for Theory and Mathematics in Biology (NITMB). Contributions by staff of NIST, an agency of the U.S. Government, are not subject to copyright protection within the United States. Certain commercial software used in this paper is identified for informational purposes only, and does not imply recommendation of or endorsement by the National Institute of Standards and Technology, nor does it imply that the products so identified are necessarily the best available for the purpose.

     \appendix
    
     \section{Spherical harmonic identities}\label{app:Ylm_id}
    In the interest of making the proofs as self-contained as possible, in this section we briefly summarize several standard properties of spherical harmonics used in the construction of certain low rank factorizations. Partly to fix notation, we start with their definition.
    \begin{definition}
        Let $\ell \ge 0$ and $-\ell \le m \le \ell.$ The spherical harmonic $Y_m^\ell$ is defined by
        \begin{align}\label{eqn:y_def}Y_\ell^m(\theta,\phi) = \sqrt{\frac{2 \ell+1}{4\pi}\frac{(\ell-|m|)!}{(\ell+|m|)!}}P_\ell^{|m|}(\cos{\theta}) e^{im\phi},\end{align}
        where $\theta \in [0,\pi]$ and $\phi \in [0,2\pi).$ Here $P_n^m$ denotes the associated Legendre function
        $$P_n^m(x) = (-1)^m (1-x^2)^{m/2} \frac{{\rm d}^m}{{\rm d}x^m} P_n(x),$$
        with $P_n$ the $n$-\emph{th} order Legendre polynomial.
    \end{definition}
    The following lemma gives a formula for expressing products of $x,y,z$ and spherical harmonics as linear combinations of spherical harmonics. Its proof follows immediately from the definition of the spherical harmonics and standard properties of associated Legendre polynomials, and is omitted.
    \begin{lemma}\label{lem:y_mult}
     Let $\ell \ge 0$ and $-\ell \le m \le n.$ If $(\phi,\theta,\rho)$ are the spherical coordinates corresponding to the point $(x,y,z)$ in Cartesian coordinates, then
     \begin{align}
     (x\pm i y) Y_\ell^m (\theta,\phi)&= -\rho \sqrt{\frac{(\ell\pm m+\frac{1}{2})^2-\frac{1}{4}}{4(\ell+1)^2-1}}Y_{\ell+1}^{m\pm 1}(\theta,\phi)+\rho \sqrt{\frac{(\ell\mp m-\frac{1}{2})^2-\frac{1}{4}}{4\ell^2-1}}Y_{\ell-
     1}^{m\pm 1}(\theta,\phi) \\
         z Y_\ell^m(\theta,\phi) &= \rho\sqrt{\frac{(\ell+1)^2-m^2}{4(\ell+1)^2-1}}Y_{\ell+1}^m(\theta,\phi)+\rho\sqrt{\frac{\ell^2-m^2}{4\ell^2-1}} Y_{\ell-1}^m(\theta,\phi)
     \end{align}
    with the convention that terms involving $Y_{\ell'}^{m'}$ with $\ell' <|m'|$ are omitted.
    \end{lemma}
     Similarly, the derivatives of the spherical harmonics with respect to the Cartesian coordinates can also be expressed again as linear combinations of spherical harmonics. We summarize these identities in the following lemma.
    \begin{lemma}\label{lem:ylm_der}
    Let $\ell \ge 0$ and $-\ell \le m \le n.$ If $(\phi,\theta,\rho)$ are the spherical coordinates corresponding to the point $(x,y,z)$ in Cartesian coordinates, then
    \begin{align}
    \partial_{\pm} Y_{\ell}^m(\theta,\phi) &= \pm\frac{\ell+1}{\rho}\sqrt{\frac{(\ell\mp m-\frac{1}{2})^2-\frac{1}{4}}{4 \ell^2-1}}Y_{\ell-1}^{m+1}(\theta,\phi)\pm\frac{\ell}{\rho}\sqrt{\frac{(\ell\mp m+\frac{3}{2})^2-\frac{1}{4}}{4(\ell+1)^2-1}}Y_{\ell+1}^{m+1}(\theta,\phi)\\
        \partial_z Y_\ell^m(\theta,\phi) &= \frac{\ell+1}{\rho}\sqrt{\frac{\ell^2-m^2}{4\ell^2-1}}Y_{\ell-1}^m(\theta,\phi) - \frac{\ell}{\rho}\sqrt{\frac{(\ell+1)^2-m^2}{4(\ell+1)^2-1}}Y_{\ell+1}^m(\theta,\phi)
    \end{align}
    with $\partial_{\pm} = \partial_x \pm i \partial_y.$
    \end{lemma}
    \begin{proof}
        The proof follows straightforwardly, if tediously, from identity (3.31) in \cite{greengard1988rapid}, which with our sign conventions becomes
        \begin{align}
            \partial_{\pm}^m \partial_z^{\ell-m} \frac{1}{\rho}=(-1)^{\ell+m}(\ell-m)!\frac{1}{\rho^{\ell+1}}P_\ell^m(\cos{\theta}) e^{\pm im \phi},\quad \ell \ge m \ge 0,
        \end{align}
     combined with the definition of the spherical harmonics (\ref{eqn:y_def}), and Lemma \ref{lem:y_mult}.
    \end{proof}
    
 \bibliography{references}

\end{document}